\documentclass{article}
\usepackage{amsmath,amssymb,amsthm,graphicx,epsfig,float,url}
\usepackage[colorlinks=true,citecolor={ForestGreen},linkcolor={Periwinkle}]{hyperref}
\usepackage{pdfsync}
\usepackage[usenames, dvipsnames]{xcolor}
\usepackage{tikz}
\usepackage{subfig}
\usepackage{bbm}
\usepackage{stmaryrd}
\usepackage{amssymb}
\usepackage{mathrsfs}  
\usepackage{caption}
\usepackage{float}
\usepackage{esint}
\usepackage{algorithm}
\usepackage{algpseudocode}

\usepackage{dsfont}
\usepackage[makeroom]{cancel}
\usetikzlibrary{patterns}
\newcommand{\R}{\textnormal{I\kern-0.21emR}}
\newcommand{\N}{\textnormal{I\kern-0.21emN}}

\renewcommand{\geq}{\geqslant}
\renewcommand{\leq}{\leqslant}

\def\B{{\mathbb B}}

\def\e{{\varepsilon}}

\allowdisplaybreaks

\def\YYint#1#2#3{{\setbox0=\hbox{$#1{#2#3}{\iint}$}
    \vcenter{\hbox{$#2#3$}}\kern-.51\wd0}}
 
\usepackage[dvipsnames]{xcolor}

\newtheorem*{theorem*}{Theorem}

\newtheorem{theorem}{Theorem}

\newtheorem{material}{material}
\newtheorem{proposition}[material]{Proposition}

\newtheorem{definition}[material]{Definition}
\newtheorem{lemma}[material]{Lemma}

\newtheorem{remark}[material]{Remark}

\def\O{{\Omega}}

\def\n{{\nabla}}

\def\p{{\varphi}}

\usepackage{xargs}
 \usepackage[colorinlistoftodos,textsize=small]{todonotes}
 \newcommandx{\christian}[2][1=]{\todo[linecolor=red,backgroundcolor=red!25,bordercolor=red,#1]{#2}}
 \newcommandx{\laura}[2][1=]{\todo[linecolor=blue,backgroundcolor=blue!25,bordercolor=blue,#1]{#2}}
 \newcommandx{\info}[2][1=]{\todo[linecolor=green,backgroundcolor=green!25,bordercolor=green,#1]{#2}}
 \newcommandx{\improvement}[2][1=]{\todo[linecolor=yellow,backgroundcolor=yellow!25,bordercolor=yellow,#1]{#2}}
 
  \newcommandx{\biblio}[2][1=]{\todo[linecolor=blue,backgroundcolor=magenta!25,bordercolor=blue,#1]{#2}}

 \numberwithin{equation}{section}

\begin{document}
\title{\sf Stability of optimal shapes and convergence of thresholding algorithms in linear and spectral optimal control problems}


\author{Antonin Chambolle\footnote{CEREMADE, UMR 7534, CNRS \& Universit\'e Paris Dauphine PSL, 75775 Paris Cedex 16, France
(\texttt{chambolle@ceremade.dauphine.fr})}
\and Idriss Mazari-Fouquer\footnote{CEREMADE, UMR 7534,  Universit\'e Paris Dauphine PSL, 75775 Paris Cedex 16, France
(\texttt{mazari@ceremade.dauphine.fr})}
\and Yannick Privat\footnote{IRMA, Universit\'e de Strasbourg, CNRS UMR 7501, Inria, 7 rue Ren\'e Descartes, 67084 Strasbourg, France (\texttt{yannick.privat@unistra.fr}).}~\footnote{Institut Universitaire de France (IUF).}}
\date{\today}

\maketitle

\begin{abstract}
We prove the convergence of the fixed-point (also called thresholding) algorithm in three optimal control problems under large volume constraints. This algorithm was introduced by C\'ea, Gioan and Michel, and is of constant use in the simulation of $L^\infty-L^1$ optimal control problems. In this paper we consider the optimisation of the Dirichlet energy, of Dirichlet eigenvalues and of certain non-energetic problems. Our proofs rely on new diagonalisation procedure for shape hessians in optimal control problems, which leads to local stability estimates.
\end{abstract}

\noindent\textbf{Keywords:} Convergence analysis, Numerical algorithms in optimal control, PDE constrained optimisation, Shape optimisation, Quantitative inequalities, Thresholding scheme.

\medskip

\noindent\textbf{AMS classification (2020): }49M05, 49M41, 49N05, 49Q10.


\section{Introduction}

\subsection{Scope of the paper: the fixed point algorithm of C\'ea,  Gioan and Michel for optimal control problems}
The fixed-point algorithm, also dubbed thresholding algorithm, is ubiquitous in the numerical simulations of optimal control problems \cite{BintzLenhart,Quantum,Chanillo2000,HintermullerKaoLaurain,KaoLouYanagida,KaoMohammadi,KaoMohammadiOsting,LLNP,Pironneau,KaoMohammadi2}. It was first introduced by C\'ea, Gioan \& Michel in the seminal \cite{Cea}. This algorithm is designed to solve optimal control problems subject to constraints expressed in $L^1$ and $L^\infty$ norms. These problems are part of a broad class of optimal control problems writing:
\begin{equation}\label{Eq:Ex}\tag{$\bold P$}
\max_{f\in \mathcal F}\left(J(f):= \int_\O j(u_f,\n u_f)\right)\text{ where }\begin{cases}-\nabla \cdot (A\n u_f)=F(x,u_f)+H(f,u_f)\,, 
\\ \mathcal Bu_f=0\end{cases}\end{equation} where $\mathcal B$ is a boundary conditions operator, $A$ is an elliptic matrix and $H$ accounts for the coupling between the state $u_f$ and the control $f$. The admissible class of controls is of the form 
\[ \mathcal F:=\left\{f\in L^\infty(\O):\quad 0\leq f\leq 1\text{ a.e., }\fint_\O f=V_0\right\}\] for some volume constraint $V_0\in (0;1)$. In most cases, no explicit description of the maximisers is available and the algorithm proposed in \cite{Cea} can be recast in terms of the so-called \emph{switch function}: provided all functionals at hand are regular enough we may represent, for a given $f\in \mathcal F$, the derivative of the criterion through a function $p_f$, called the \emph{switch function of the optimal control problem}: for any admissible perturbation $h$ at $f$ (for the sake of simplicity, such that for any $\e>0$ small enough $f+\e h\in \mathcal F$;  see Definition \ref{De:Admissible} for the proper notion), the derivative of $J$ at $f$ in the direction $h$ writes
\[ \dot J(f)[h]=\int_\O h p_f.\]

The thresholding algorithm, in short, picks the direction $h$ that maximises $\dot J(f)[\cdot]$:
\begin{algorithm}[H]
\caption{Thresholding algorithm}\label{Algo}
\begin{algorithmic}[1]
\State Initialisation at $f_0\in \mathcal F$
\State $k\gets 0$
\State Compute $p_{f_k}$
\State Compute $c_k$ such that $\mathrm{Vol}(\{p_{f_k}>c_k\})=V_0\mathrm{Vol}(\O)$.
\State $f_k\gets \mathds 1_{\{p_{f_k}>c_k\}}$
\State $k\gets k+1$.
\end{algorithmic}
\end{algorithm}
Some difficulties can arise. Indeed, we can only hope for this algorithm to be well-posed if the solutions to the optimal control problem \eqref{Eq:Ex} only has so-called bang-bang solutions \emph{i.e.} any optimal $f^*$ writes $f^*=\mathds 1_{E^*}$ for some measurable $E^*$. Indeed, in general, one can only define a Lagrange multiplier $c_k$ such that 
\[ \mathrm{Vol}(\{p_{f_k}>c_k\})\leq V_0\mathrm{Vol}(\O)< \mathrm{Vol}(\{p_{f_k}\geq c_k\}).\] For the algorithm to be properly defined, one needs to ensure that the level-set $\{p_{f_k}=c_k\}$ has zero measure, which ``usually" implies that \eqref{Eq:Ex} indeed satisfies this bang-bang property. Observe that some simple linear control problems do not satisfy this property \cite{Nadin_2020}. But even assuming this is not a problem, and despite the fact that this method has proved very efficient in past contributions, the theoretical convergence of the sequence of iterates was, to the best of our knowledge, never proved apart from cases where explicit computations are available; see \cite{KaoMohammadiOsting} where a detailed study of the convergence rate of the method in the one-dimensional case is undertaken. In \cite[Condition (2.4)]{Cea} a sufficient condition for the convergence of the algorithm is given for very specific types of functionals; some of the conditions given in this paper that ensure convergence are similar and are akin to a coercivity condition in shape differentiation \cite{DambrineLamboley}. Finally, we mention that in \cite{Eppler_2007} a study of convergence rate for discretised version of shape optimisation algorithms is obtained.

The purpose of this paper is to study Algorithm \ref{Algo} for three optimal control problems and to \emph{establish its global convergence for large volume constraints} $V_0$. This is done by building on recent progress in the study of quantitative inequalities in optimal control theory \cite{MazariNLA,MRB2020}, and by introducing a \emph{new diagonalisation procedure for shape hessians in optimal control theory}. Such a procedure is of independent interest and relates to stability estimates in optimal control theory.

\paragraph{Plan of the article }In section \ref{Se:IntroPreliminaire} we lay out our notations and introduce our admissible classes.
We study three different problems; the introduction is divided accordingly.
\begin{itemize}
\item First, we consider the optimisation of the Dirichlet energy, that is, the optimisation problem
\[ \min_{f\in \mathcal F}\frac12\int_\O |\n u_f|^2-\int_\O fu_f\text{ subject to }\begin{cases}-\Delta u_f=f&\text{ in }\O\,, 
\\ u_f=0&\text{ on }\partial \O.\end{cases}\] We refer to section \ref{Se:IntroDirichlet} of this introduction. 
\item Second, we consider the optimisation of weighted Dirichlet eigenvalues, that is, the optimisation problem
\[ \min_{f\in \mathcal F}\lambda(f)\text{ where $\lambda(f)$ is the first Dirichlet eigenvalue of $-\Delta -f$}\] We underline that this is a bilinear control problem. We refer to Section \ref{Se:IntroEigenvalue}.
\item Third, we consider a non-energetic optimal control problem: for a given function $j=j(x,u)$ that is convex in $u$ we seek to solve
\[\max_{f\in \mathcal F}J(f):=\int_\O j(x,u_f)  \text{ subject to }\begin{cases}-\Delta u_f=f&\text{ in }\O\,, 
\\ u_f=0&\text{ on }\partial \O.\end{cases}\] We refer to section \ref{Se:IntroNonEnergetic}.
\end{itemize}
For the sake of readability, we give in section \ref{Se:IntroPlan} the general plan of the proof of convergence. The core of the paper is devoted to the proofs of the result for the Dirichlet energy. As the proofs for the other two problems are very similar we postpone them to appendices.
In the conclusion of this article, we discuss several limits of our analysis and offer possible future research directions. 
 
 \paragraph{Related algorithms.}
 Before we move on to the main part of the introduction, let us mention that Algorithm \ref{Algo} is linked to the \emph{thresholding algorithm}, studied in detail since its introduction in \cite{MBO} and which, broadly speaking, seeks to approximate mean-curvature motions using a thresholding of the solutions to certain PDEs. The theoretical and numerical aspects of this scheme have been the subject of an intense research activity in the past years \cite{Barles_1995,evans1993convergence,ishii1995generalization,ISHII_1999,laux2016convergence,Laux_2017,ruuth2001diffusion,Swartz_2017}. Let us mention that most of the methods used in the aforementioned works can not be used here as they use the behaviour of solutions of PDEs defined in $\R^d$. The presence of boundary conditions in the models under consideration here prohibits relying on the tools these authors developed .

Finally, it should be mentioned that a closely related fixed-point type algorithm was used in \cite{MR2235384,MR2837788} to numerically solve topological optimization problems.

\subsection{Notational conventions, preliminary definitions and setting}\label{Se:IntroPreliminaire}
Throughout the paper, $\O$ is a fixed $\mathscr C^2$ bounded domain of $\R^d$ ($d\geq 2$) which is simply connected (in particular $\partial \O$ is a connected smooth submanifold of $\R^d$). For a fixed volume constraint $V_0\in (0;1)$ we define the set of admissible controls 
\begin{equation}\tag{$\bold{Adm}$}
\mathcal F(V_0):=\left\{f\in L^\infty(\O):\, 0\leq f\leq 1\text{ a.e., } \fint_\O f=V_0\right\}.
\end{equation}
By \cite[Proposition 7.2.17]{HenrotPierre} the set $\mathcal F(V_0)$ is convex and closed for the weak $L^\infty-*$ topology. Moreover, its extreme points are characteristic functions of subsets: using $\mathrm{Extr}(X)$ to denote the extreme points of a convex set $X$ we have 
\[ \mathrm{Extr}(\mathcal F(V_0))=\left\{\mathds 1_E\,, E\subset \O\,, \mathrm{Vol}(E)=V_0\mathrm{Vol}(\O)\right\}.\] It is convenient to introduce a bit of terminology:
\begin{definition}\label{De:BangBang}
A bang-bang function of $\mathcal F(V_0)$ is an extreme point of $\mathcal F(V_0)$. 
\end{definition}
In other words, $m$ is bang-bang in $\mathcal F(V_0)$ if, and only if, there exists a measurable subset $E\subset \O$ such that $\mathrm{Vol}(E)=V_0\mathrm{Vol}(\O)$ and such that $m=\mathds 1_E$.

\begin{remark}\label{Re:MaxBangbang}
A salient feature of all the optimisation problems we consider in this paper is that their solutions are bang-bang functions. We explain in the conclusion why we can not at this point bypass this underlying assumption.
\end{remark}
Finally, as our proofs rely on optimality conditions, we need to introduce the tangent cone to an element $f\in \mathcal F(V_0)$. 
\begin{definition}\label{De:Admissible}
Let $f \in \mathcal F(V_0)$. The tangent cone to $\mathcal F(V_0)$ at $f$ is the set of functions $h\in L^\infty(\O)$ such that, for any sequence $\{t_k\}_{k\in \N}$ converging to 0, there exists a sequence $\{h_k\}_{k\in \N}$ that converges strongly to $h$ in $L^2(\O)$ and such that, for any $k\in \N$, $f+t_kh_k\in \mathcal F(V_0)$. Functions belonging to this tangent cone are called admissible perturbations at $f$.
\end{definition}
Further characterisations of the tangent cone are available in \cite{bednarczuk:inria-00073499,CominettiPenot}.

\subsection{Optimisation of the Dirichlet energy}\label{Se:IntroDirichlet}
\paragraph{The energy functional and the optimal control problem.}
The first problem we study is the optimisation of the Dirichlet energy. For a given $f\in L^2(\O)$, let $u_f\in W^{1,2}_0(\O)$ be the unique solution of 
\begin{equation}\label{Eq:Main}\begin{cases}
-\Delta u_f=f&\text{ in }\O\,, 
\\ u_f=0&\text{ on }\partial \O.\end{cases}\end{equation} This equation has a variational formulation: $u_f$ is the unique minimiser in $W^{1,2}_0(\O)$ of the energy $\mathcal E_f$ defined as 
\[ \mathcal E_f:W^{1,2}_0(\O)\ni u\mapsto \frac12\int_\O |\n u|^2-\int_\O fu.\]
The \emph{Dirichlet energy} associated with $f$ is defined by
\begin{equation}\label{Eq:DirichletEnergy} \mathscr E(f):=\mathcal E_f(u_f)=\frac12\int_\O |\n u_f|^2-\int_\O fu_f.\end{equation} The optimal control problem is 
\begin{equation}\label{Pv:Dirichlet}\tag{$\bold{P}_{\mathrm{Dir}}$}
\fbox{$\displaystyle \min_{f\in \mathcal F(V_0)}\mathscr E(f).
$}
\end{equation}

The following result is standard.
\begin{lemma}\label{Le:BasicDirichlet}
The problem \eqref{Pv:Dirichlet} has a solution. Furthermore, the map $\mathscr E$ is strictly concave. In particular any solution $f^*$ of \eqref{Pv:Dirichlet} is a bang-bang function in the sense of Definition \ref{De:BangBang}: there exists $E^*\subset \O$ such that $f^*=\mathds 1_{E^*}$, with $\mathrm{Vol}(E^*)=V_0\mathrm{Vol}(\O).$ 
\end{lemma}

\paragraph{The switch function and the thresholding algorithm.}
To describe the thresholding algorithm we need to specify the switch function of $\mathscr E$. The following lemma is well-known.
\begin{lemma}\label{Le:SwitchDirichlet}
The map $f\mapsto \mathscr E(f)$ is Fr\'echet differentiable at any $f\in \mathcal F(V_0)$ and, for any $f\in \mathcal F(V_0)$, for any admissible perturbation $h$, there holds
\[ \dot{\mathscr E}(f)[h]=-\int_\O h u_f.\]
\end{lemma}
Consequently, for the Dirichlet energy, the switch function of the optimal control problem is  $-u_f$. {This allows to describe the thresholding algorithm:
\begin{algorithm}[H]
\caption{Thresholding algorithm for the Dirichlet energy}\label{Algo:Dirichlet}
\begin{algorithmic}[1]
\State Initialisation at $f_0\in \mathcal F$
\State $k\gets 0$
\State Compute $u_{f_k}$
\State Compute $c_k$ such that $\mathrm{Vol}(\{u_{f_k}>c_k\})=V_0\mathrm{Vol}(\O)$.
\State $f_k\gets \mathds 1_{\{u_{f_k}>c_k\}}$
\State $k\gets k+1$.
\end{algorithmic}
\end{algorithm}}

\medskip
\begin{remark}\label{Re:DirichletWell?}
Of course we have the same problem as in Algorithm \ref{Algo}, namely, that we would need to ensure that the algorithm is well-defined in the sense that for any index $k$ there indeed exists a $c_k\in \R$ such that $\mathrm{Vol}(\{u_{f_k}>c_k\})=V_0\mathrm{Vol}(\O)$. This is also covered by our theorem.
\end{remark}

\paragraph{Some terminology.}
 We introduce the following definition:
\begin{definition}[Critical points]\label{De:CriticalPoint}
For any $V_0\in (0;1)$, for any $f\in \mathcal F(V_0)$, we say $f$ is a \textbf{critical point of $\mathscr E$} if:
\begin{enumerate}
\item There exists a unique $\mu_f=\mu_{u_f,V_0}$ such that 
\[ \mathrm{Vol}(\{u_f>\mu_f\})=V_0\mathrm{Vol}(\O)=\mathrm{Vol}(\{u_f\geq \mu_f\}),\]
\item $f=\mathds 1_{\{u_f>\mu_f\}}.$
\end{enumerate}
A \textbf{critical set} is a subset $E$ of $\O$ such that $\mathds 1_E\in \mathcal F(V_0)$ and such that $\mathds 1_E$ is a critical point of $\mathscr E$.
\end{definition}

It is natural to expect that the algorithm will converge to a local minimiser of the functional. In general, one must clarify the meaning behind ``local minimiser"; more than mere local minimisers we use the notion of ``stable local minimisers". Our definition here is the following:
\begin{definition}\label{De:LocalMinimiserDirichlet}
An admissible control $f^*\in \mathcal F(V_0)$ is called a stable local minimiser of $\mathscr E$ in $\mathcal F(V_0)$ if there exist $C=C(f^*)\,, \delta=\delta(f^*)>0$ such that 
\[ \forall f\in \mathcal F(V_0)\,, \Vert f-f^*\Vert_{L^1(\O)}\leq \delta(f^*)\Rightarrow \mathscr E(f)\geq \mathscr E(f^*)+C(f^*)\Vert f-f^*\Vert_{L^1(\O)}^2.\]
\end{definition}

\paragraph{Main result}
Our main result is the following:
\begin{theorem}\label{Th:Dirichlet}
There exists $\e>0$ such that, if $1-\e\leq V_0<1$, for any initialisation $f_0\in \mathcal F(V_0)$, the sequence $\{f_k\}_{k\in \N}$ generated by Algorithm \ref{Algo:Dirichlet} converges strongly in $L^1(\O)$ to a stable local minimiser of $\mathscr E$ in $\mathcal F(V_0)$.
\end{theorem}


\subsection{Optimisation of weighted eigenvalues}\label{Se:IntroEigenvalue}
\paragraph{The energy functional and the optimal control problem.}
The second functional under consideration is the principal eigenvalue of a Dirichlet operator. For any $f\in \mathcal F(V_0)$, let $\lambda(f)$ be the first eigenvalue of the operator $-\Delta-f$ endowed with Dirichlet boundary conditions. The variational formulation of $\lambda(f)$ is
\begin{equation}\label{Eq:Rayleigh}
\lambda(f)=\min_{u\in W^{1,2}_0(\O)\backslash\{0\}}\frac{\int_\O|\n u|^2-\int_\O fu^2 }{\int_\O u^2}.\end{equation} It is standard to see that this eigenvalue is simple and that any associated eigenfunction has constant sign. Thus, up to normalisation, the eigenpair $(\eta_f,\lambda(f))$ satisfies
\begin{equation}
\label{Eq:Eigenpair}
\begin{cases}
-\Delta \eta_f=\lambda(f)\eta_f+f\eta_f&\text{ in }\O\,, 
\\ \eta_f\in W^{1,2}_0(\O)\,, 
\\ \eta_f>0&\text{ in }\O\,, 
\\ \int_\O \eta_f^2=1.
\end{cases}
\end{equation}
The optimal control problem under consideration is 
\begin{equation}\label{Pv:Eigenvalue}\tag{$\bold{P}_{\mathrm{Eigen}}$}
\fbox{$\displaystyle \min_{f\in \mathcal F(V_0)}\lambda(f).
$}
\end{equation}
This problem has been studied at length, see \cite[Chapter 8]{henrot2006} or the more recent \cite{LLNP}. Part of the recent interest in it is due to its applications in spatial ecology, as the sign of $\lambda(f)$ predicts extinction or survival of a species, see \cite[Introduction]{MazariThese} and the references therein. The thresholding algorithm was applied to this problem in \cite{HintermullerKaoLaurain,KaoLouYanagida,LLNP}.

Similar to Lemma \ref{Le:BasicDirichlet}, we have the following result.

\begin{lemma}\label{Le:BasicEigenv}
The problem \eqref{Pv:Eigenvalue} has a solution. Furthermore, the map $\lambda$ is strictly concave. In particular any solution $f^*$ of \eqref{Pv:Eigenvalue} is a bang-bang function in the sense of Definition \ref{De:BangBang}: there exists $E^*\subset \O$ such that $f^*=\mathds 1_{E^*}$, with $\mathrm{Vol}(E^*)=V_0\mathrm{Vol}(\O).$ 
\end{lemma}

\paragraph{The switch function and the thresholding algorithm.}
Let us now give the switch function of \eqref{Pv:Eigenvalue}.
\begin{lemma}\label{Le:SwitchEigenvalue}
The map $f\mapsto \lambda(f)$ is Fr\'echet differentiable at any $f\in \mathcal F(V_0)$ and, for any $f\in \mathcal F(V_0)$, for any admissible perturbation $h$ at $f$, there holds
\[ \dot{\lambda}(f)[h]=-\int_\O h \eta_f^2.\]
\end{lemma}
Consequently, for the Dirichlet eigenvalue, the switch function is $-\eta_f^2$. Note that as $\eta_f>0$,  computing a level-set of $\eta_f$ is the same as computing a level-set of $\eta_f^2$. This allows to describe the thresholding algorithm for \eqref{Pv:Eigenvalue}.
\begin{algorithm}[H]
\caption{Thresholding algorithm for the Dirichlet eigenvalue}\label{Algo:Eigenvalue}
\begin{algorithmic}[1]
\State Initialisation at $f_0\in \mathcal F$
\State $k\gets 0$
\State Compute $\eta_{f_k}$
\State Compute $c_k$ such that $\mathrm{Vol}(\{\eta_{f_k}>c_k\})=V_0\mathrm{Vol}(\O)$.
\State $f_k\gets \mathds 1_{\{\eta_{f_k}>c_k\}}$
\State $k\gets k+1$.
\end{algorithmic}
\end{algorithm}

Remark \ref{Re:DirichletWell?} about the well-posedness of the algorithm applies here as well.

 To alleviate the presentation, we give the precise definitions of  ``critical point" and of ``local minimisers" at the beginning of the proof; they are identical to Definitions \ref{De:CriticalPoint}-\ref{De:LocalMinimiserDirichlet}, up to replacing  $\mathscr E$ with $\lambda$.

\paragraph{Main result}
The main theorem is the following:

\begin{theorem}\label{Th:Eigenvalue}
There exists $\e>0$ such that, if $1-\e\leq V_0<1$, for any initialisation $f_0\in \mathcal F(V_0)$, the sequence $\{f_k\}_{k\in \N}$ generated by Algorithm \ref{Algo:Eigenvalue} converges strongly in $L^1(\O)$ to a local minimiser of $\lambda$ in $\mathcal F(V_0)$.
\end{theorem}
The proof is very similar to that of Theorem \ref{Th:Dirichlet}; it is given in \cite[Section A]{CMFPSP}.

We conclude this section with an illustration of the thresholding algorithm~\ref{Algo:Eigenvalue} for optimizing the functional $\lambda$. In Figures~\ref{FigDir03} and \ref{FigDir08}, we show the optimal domain, as well as the evolution of the quantity $\Vert f_{k+1}-f_k\Vert_{L^1(\Omega)}$, using the same notation as in algorithm~\ref{Algo:Eigenvalue}, which is used as a stopping criterion for the algorithm, as a function of the iterations. We observe that the number of iterations required for the stopping criterion to be below the fixed tolerance (in this case, $10^{-6}$ for these figures) is very small. The numerical results also suggest that the statement of Theorem~\ref{Th:Eigenvalue} is likely to be valid for large values of $\varepsilon$ close to 1.
 	\begin{figure}[h!]
	\centering
	\includegraphics[height=5cm]{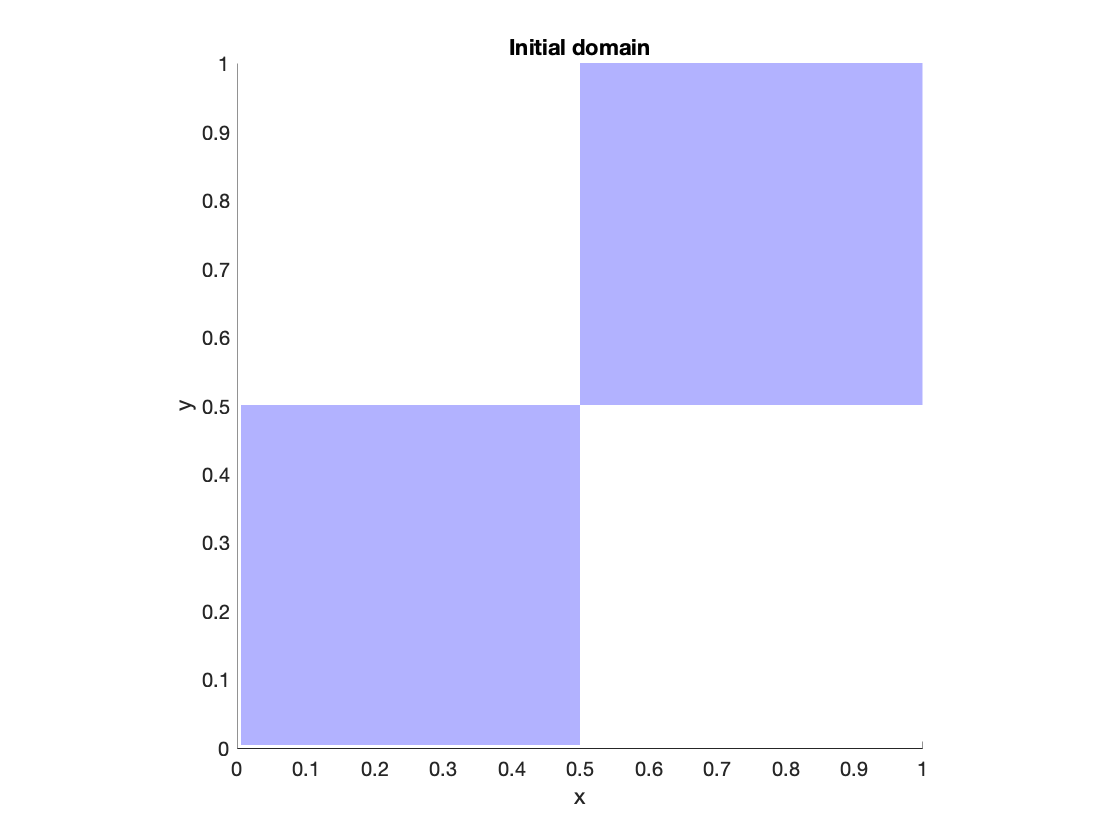}
	\caption{Domain $E_0$ chosen to initialize algorithm~\ref{Algo:Eigenvalue}, i.e., $f_0=\mathds{1}_{E_0}$.. \label{FigDir03}}
	\end{figure}
	\begin{figure}[h!]
	\centering
	\includegraphics[height=5cm]{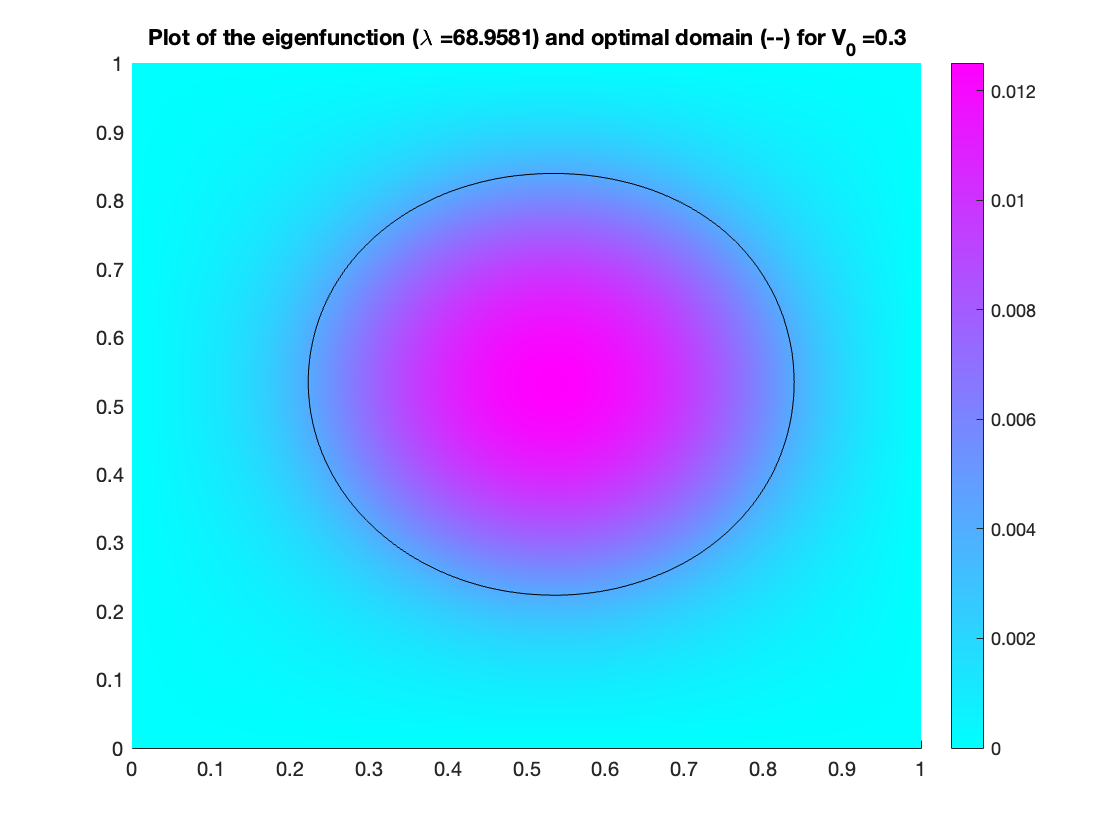}
	\includegraphics[height=5cm]{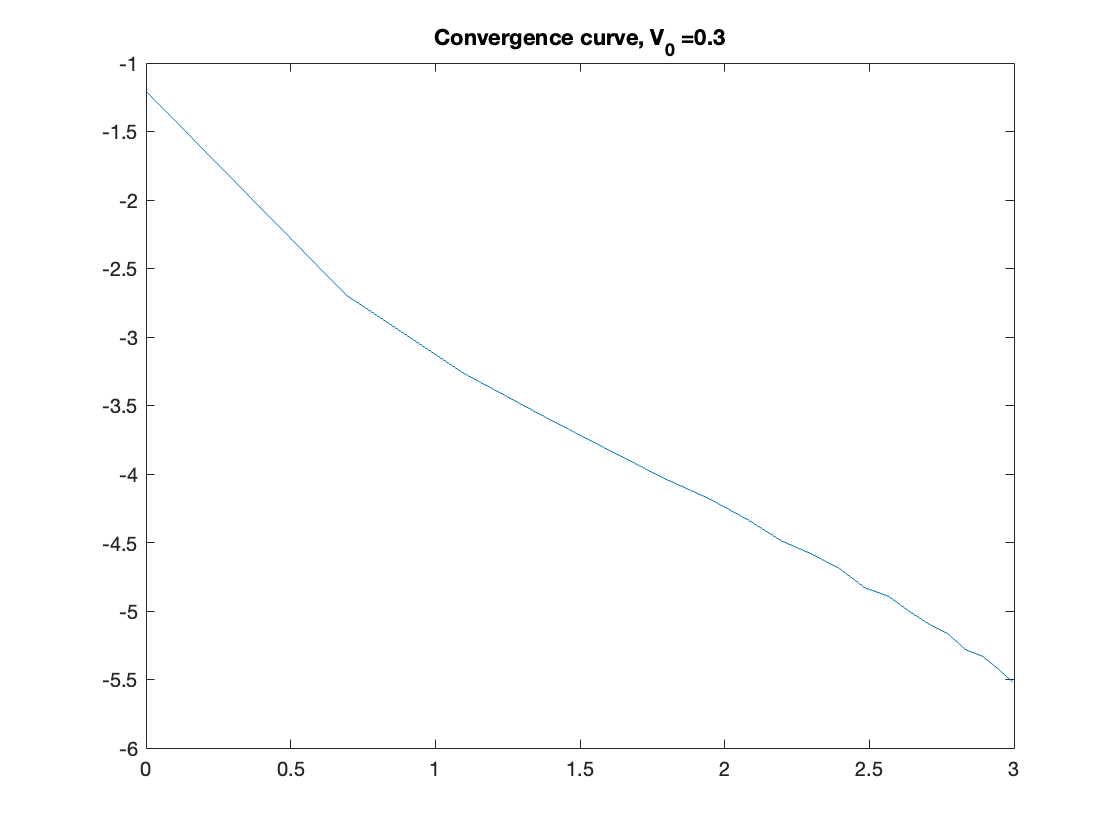}
	\caption{$V_0=0.3$. Left: representation of the optimal domain (black contour) and the eigenfunction (warmer colors indicate where the function takes its maximum values). Right: evolution of the quantity $\Vert f_{k+1}-f_k\Vert_{L^1(\Omega)}$ as a function of iterations, on a logarithmic scale. \label{FigDir03}}
	\end{figure}
	\begin{figure}[h!]
	\centering
	\includegraphics[height=5cm]{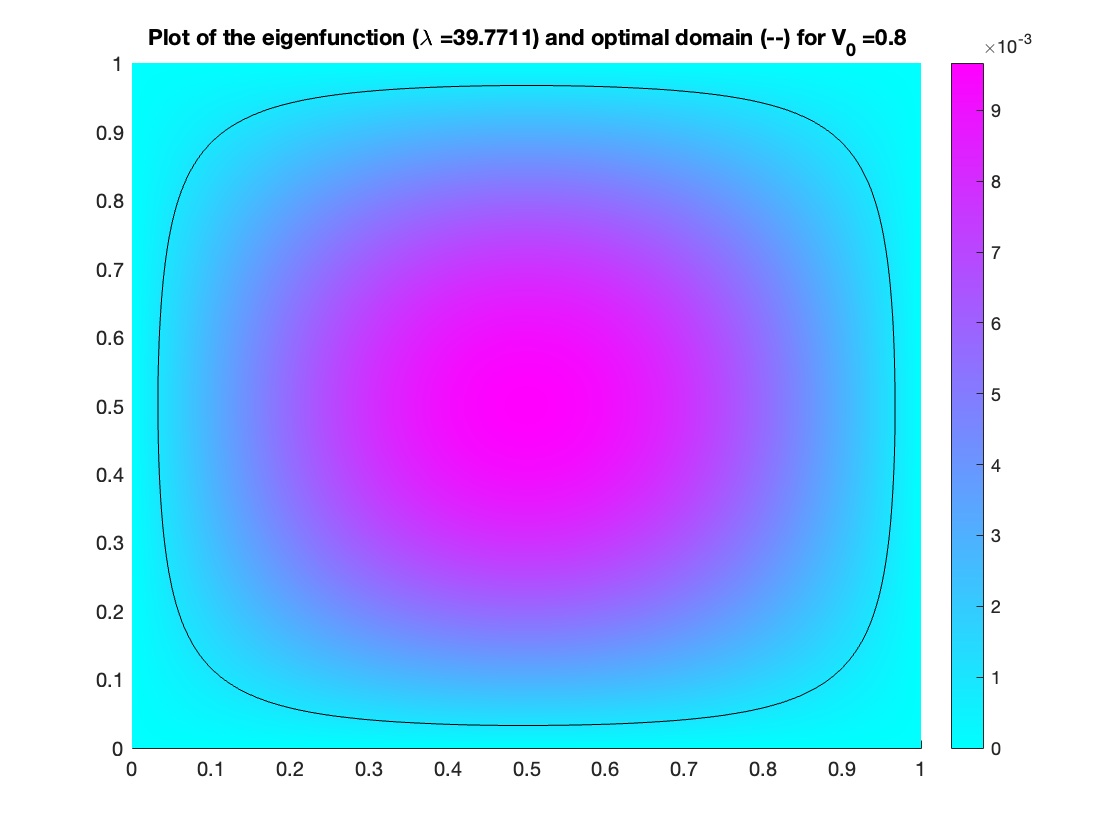}
	\includegraphics[height=5cm]{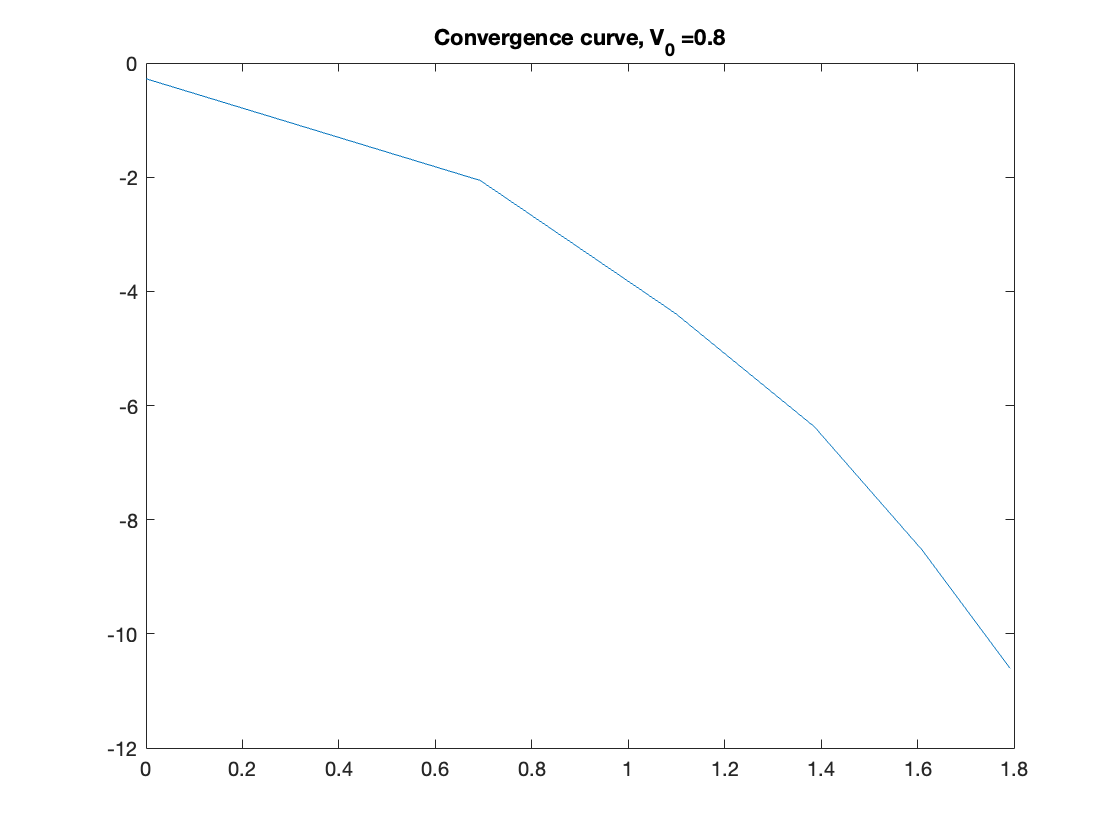}
	\caption{$V_0=0.8$. Left: representation of the optimal domain (black contour) and the eigenfunction (warmer colors indicate where the function takes its maximum values). Right: evolution of the quantity $\Vert f_{k+1}-f_k\Vert_{L^1(\Omega)}$ as a function of iterations, on a logarithmic scale. \label{FigDir08}}
	\end{figure}

\subsection{Optimisation of non-energetic criteria}\label{Se:IntroNonEnergetic}
\paragraph{The  functional and the optimal control problem.}
The third and final problem we consider is the optimisation of a non-energetic criterion. We consider, for any $f\in \mathcal F(V_0)$, the solution $u_f$ of the Dirichlet problem \eqref{Eq:Main} (the same $u_f$ we used in the study of the Dirichlet energy). We fix a non-linearity $j=j(x,u)$ satisfying the following regularity assumptions:
\begin{equation}\label{Hyp:j}\tag{$\bold{H}_j$}
\begin{cases}
j\text{ is $\mathscr C^2$ in $u$, $\mathscr C^1$ in $x$}\,, 
\\  \text{For $k=0,1,2,i=0,1$, for any compact $K\subset\joinrel\subset\R$ there holds }\underset{\overline\O\times K}\sup \Vert \partial_{x^i}^i \partial^k_{u^k}j\Vert_{L^\infty(\O)}<\infty\,, 
\\ \text{$j$ is strictly convex in $u$ and for any $u>0$ $\underset{\overline\O}\inf \partial^2_{u^2} j(x,u)>0$,}
\\ \text{$j$ is increasing in $u$ for any $x\in \overline \O$, for any $u>0$, $\partial_uj(x,u)>0$.}
\end{cases}
\end{equation}
The problem under consideration is 
\begin{equation}\tag{$\bold{P}_j$}\label{Pv:NonEnergetic}
\fbox{$\displaystyle \max_{f\in \mathcal F(V_0)}\left(J(f):=\int_\O j(x,u_f(x))dx\right).
$}
\end{equation}

We start with a standard result.
\begin{lemma}\label{Le:BasicNonEnergetic}
The problem \eqref{Pv:NonEnergetic} has a solution. Furthermore, the map $J$ is strictly convex. In particular any solution $f^*$ of \eqref{Pv:Eigenvalue} is a bang-bang function in the sense of Definition \ref{De:BangBang}: there exists $E^*\subset \O$ such that $f^*=\mathds 1_{E^*}$, with $\mathrm{Vol}(E^*)=V_0\mathrm{Vol}(\O).$ 
\end{lemma}

\paragraph{The switch function and the thresholding algorithm.}
We now describe the switch function in the following standard lemma:
\begin{lemma}\label{Le:SwitchNonEnergetic}
The map $J$ is Fr\'echet-differentiable and, for any $f\in \mathcal F(V_0)$, for any admissible perturbation $h$ at $f$ there holds
\[ \dot J(f)[h]=\int_\O h p_f\]where $p_f$ is the unique solution of \begin{equation}\label{Eq:AdjointNonEnergetic}\begin{cases}-\Delta p_f=\partial_u j(x,u_f)&\text{ in }\O\,, 
\\ p_f\in W^{1,2}_0(\O).
\end{cases}
\end{equation}
\end{lemma}
Lemma \ref{Le:SwitchNonEnergetic} states that  the function $p_f$ defined in \eqref{Eq:AdjointNonEnergetic} is the switch function of the optimal control problem \eqref{Pv:NonEnergetic}. 
This allows to describe the thresholding algorithm for \eqref{Pv:NonEnergetic}.

\begin{algorithm}[H]
\caption{Thresholding algorithm for non-energetic problems eigenvalue}\label{Algo:NonEnergetic}
\begin{algorithmic}[1]
\State Initialisation at $f_0\in \mathcal F$
\State $k\gets 0$
\State Compute $p_{f_k}$
\State Compute $c_k$ such that $\mathrm{Vol}(\{p_{f_k}>c_k\})=V_0\mathrm{Vol}(\O)$.
\State $f_k\gets \mathds 1_{\{p_{f_k}>c_k\}}$
\State $k\gets k+1$.
\end{algorithmic}
\end{algorithm}
Remark \ref{Re:DirichletWell?} applies here as well.
The notion of ``critical point" and of ``local minimisers", which were defined for the Dirichlet energy in Definitions \ref{De:CriticalPoint}-\ref{De:LocalMinimiserDirichlet}, are identical for non-energetic criteria: it suffices to replace  $\mathscr E$ with $J$.

\paragraph{Main result.} The main result regarding non-energetic problems is the following convergence theorem:
\begin{theorem}\label{Th:NonEnergetic}
Assume $j$ satisfies \eqref{Hyp:j}. There exists $\e>0$ such that, if $1-\e\leq V_0<1$, for any initialisation $f_0\in \mathcal F(V_0)$, the sequence $\{f_k\}_{k\in \N}$ generated by Algorithm \ref{Algo:NonEnergetic} converges strongly in $L^1(\O)$ to a local minimiser of $J$ in $\mathcal F(V_0)$.
\end{theorem}
As it is very similar to the proof of Theorem \ref{Th:Dirichlet}, the proof of Theorem \ref{Th:NonEnergetic} is given in \cite[Section B]{CMFPSP}.

%
%
%
%
\subsection{The plan for the proof of convergence}\label{Se:IntroPlan}
As the schemes of proof for Theorems \ref{Th:Dirichlet}, \ref{Th:Eigenvalue} and \ref{Th:NonEnergetic} are identical and as they are each quite long we now present the general plan. This allows use to single out the relevant elements for each step of the proof. 

\paragraph{Consequence of the convexity of the functional.}
The first part of the proof of convergence is to use the convexity or concavity properties of the functional (depending on whether we are maximising or minimising). To be more precise, consider a generic functional $\mathcal G:\mathcal F(V_0)\to \R$ that we seek to minimise in $\mathcal F(V_0)$. Assume $\mathcal G$ is Fr\'echet-differentiable and denote by $-q_f$ its switch function at a given $f\in \mathcal F(V_0)$ so that, for any $f$, for any admissible perturbation $h$ at $f$, there holds 
\[ \dot{\mathcal G}(f)[h]=-\int_\O q_f h.\]  Starting from $f_k$, we choose $f_{k+1}:=\mathds1_{\{q_{f_k}>c_k\}}$ where $\mathrm{Vol}\left(\{q_{f_k}>c_k\}\right)=V_0\mathrm{Vol}(\O)$. If the functional $\mathcal G$ is concave, then we have the estimate
\begin{equation}\label{Eq:ConsequenceConvexity} \mathcal G(f_{k+1})-\mathcal G(f_k)\leq \dot{\mathcal G}(f_k)(f_{k+1}-f_k)=-\int_\O q_{f_k}\left(f_{k+1}-f_k\right).\end{equation} Thus we can be satisfied with gaining enough control of the first order derivative, which is a linear (in the perturbation $h$) functional.

\paragraph{Consequence of large volume constraints I: quantitative bathtub principle and well-posedness of thresholding algorithms.}
This is where the large volume constraint plays a first role. Keeping the same notations as in the previous paragraph, observe that from Lemmata \ref{Le:SwitchDirichlet}-\ref{Le:SwitchEigenvalue}-\ref{Le:SwitchNonEnergetic} the switch function $q_f$ satisfies a PDE with Dirichlet boundary conditions; as the PDEs involved enjoy a maximum principle, it is expected that, if $V_0$ is close enough to 1, then the set $\{q_f>c\}$ such that $\mathrm{Vol}(\{q_f>c\})=V_0$ is close to the boundary, and even $\mathscr C^1$ close to $\partial \O$. Moreover, it is expected that by elliptic regularity $q_f$ is $\mathscr C^{1,\alpha}$ close to $q_{\overline f\equiv 1}$, the function associated to $f=1$ (observe that when $V_0=1$ the set of admissible controls is reduced to a single point: $\mathcal F(V_0)=\{\overline f\equiv 1\}$).  Applying the Hopf Lemma to $q_{\overline f}$ we may get that, for $V_0$ close to 1, the gradient of $q_f$ is always non-zero in a neighbourhood of $\partial \O$. With these two informations combined, we deduce that the level set $\{q_f=c_f\}$ has zero measure and, hence,  that the thresholding algorithm is indeed well defined.

But there is another crucial information we can deduce, and it is that $|\n q_f|\neq 0$ on $\{q_f=c_f\}$. Combined with the aforementioned regularity of the boundary we are then in a position to apply a \emph{quantitative bathtub principle} \cite[Proposition 26]{MRB2020} (see also the related \cite{Cianchi2008}). One information we have not yet given is that the choice of the next iterate $f_{k+1}$ in the thresholding algorithm is that $f_{k+1}$ is actually the unique\footnote{using the information about the measure of level sets that were already derived.} maximiser of the functional 
\[T:\mathcal F(V_0)\ni f\mapsto \mapsto \int_\O q_{f_k}f.\] With this notation \eqref{Eq:ConsequenceConvexity} rewrites $\mathcal G(f_{k+1})-\mathcal G(f_k)\leq T(f_k)-T(f_{k+1})$. The quantitative bathtub principle writes: under certain regularity assumptions, for a certain constant $C>0$,
\[\forall f\in \mathcal F(V_0)\,, T(f)-T(f_{k+1})\leq -C\Vert f_{k+1}-f\Vert_{L^1(\O)}^2.
\]Assuming that the constant $C$ is uniform in $k$ this would yield 
\[ G(f_{k+1})-G(f_k)\leq -C\Vert f_{k+1}-f_k\Vert_{L^1(\O)}^2\] and, in turn, prove that the sequence $\{f_k\}_{k\in \N}$ satisfies $\sum_{k\in \N}\Vert f_{k+1}-f_k\Vert_{L^1(\O)}^2<\infty$. Finally, we can prove that this implies that, either the sequence $\{f_k\}_{k\in \N}$ converges, or it has an infinite number of closure points. To rule the latter possibility out, we use a shape derivation argument.

\paragraph{Consequence of large volume constraints II: shape derivative formalism and coercivity of shape hessians.}
We use the shape derivative formalism of Hadamard; see \cite[Chapter 5]{HenrotPierre} for a full introduction. We can prove that any closure point of the sequence $\{f_k\}_{k\in \N}$ is a bang-bang function $f^*=\mathds 1_E$ that is also a critical point (in the sense that $\{q_{f^*}>c_{f^*}\}=E^*$). With an abuse of notation, for a bang-bang function $f=\mathds 1_E$, we may define the shape functional 
\[\mathcal G(E):=\mathcal G(\mathds 1_E).\]
To show that possible closure points are isolated, which would allow to conclude, we prove that these closure points are stable in the sense of shapes \cite{DambrineLamboley}. Namely, considering a critical shape $E$ and smooth enough vector fields $\Phi$, we define shape Lagrangians (with a Lagrange multiplier $\mu=\mu(E^*)$ independent of $\Phi$ and encoded by first order optimality conditions)
\[ g_\Phi(t):=\mathcal G\left((\mathrm{Id}+t\Phi)E\right)+\mu \mathrm{Vol}\left((\mathrm{Id}+t\Phi)E\right)\] so that $g'(0)=0$. We then show a coercivity of second order shape derivatives or, in other words, that there exists a constant $C>0$ such that for any $\Phi$
\[ g_\Phi''(0)\geq C \Vert \langle \Phi,\nu\rangle\Vert_{X(\partial E)}^2\] for some norm $X$. In this case we say the second-order derivative is $X$-coercive. Our main contribution in this paper is to find a procedure to diagonalise these second order shape derivatives by using a new interior Steklov system of eigenvalues and eigenfunctions and to prove that, under large volume constraints, these second-order shape derivatives are indeed coercive with respect to the $L^2$-norm.  We believe this is of independent interest. Once this is done we can apply the general procedure of \cite{DambrineLamboley}, adapted in \cite{MazariQuantitative,MazariNLA,MRB2020} to the setting of optimal control problems, to prove that critical points are isolated. 

\section{Proof of Theorem \ref{Th:Dirichlet}}
We begin with the study of \eqref{Pv:Dirichlet}.
\subsection{First consequences of a large volume constraint}
We gather here several crucial properties of the problem under large enough volume constraints. First, we show that when $V_0$ is close enough to 1 the level-sets generated by the thresholding algorithm enjoy strong regularity properties.

To state our results in a synthetic way we let, for any $f\in \mathcal F(V_0)$, $\mu_{f,V_0}$ be the unique real number such that 
\[ \frac{\mathrm{Vol}(\{u_f>\mu_{f,V_0}\})}{\mathrm{Vol}(\O)}\leq V_0\leq  \frac{\mathrm{Vol}(\{u_f\geq \mu_{f,V_0}\})}{\mathrm{Vol}(\O)}.\]
We also define 
\[ \omega_{f,V_0}:=\{ u_f>\mu_{f,V_0}\}.\]
\begin{lemma}\label{Le:V0Lagrange}
There exists $\e_0>0$ and $\delta_0>0$ such that, for any $V_0\in (1-\e_0;1)$, for any $f\in \mathcal F(V_0)$, we have on the one hand
\[ \mathrm{Vol}(\{ u_f=\mu_{f,V_0}\})=0\] and, on the other hand, the following regularity properties:
\begin{enumerate}
\item $\omega_{f,V_0}$ has a $\mathscr C^{1,\alpha}$ boundary (for any $\alpha \in (0;1)$),
\item
\[ \min_{\partial \omega_{f,V_0}} \left| \frac{\partial u_f}{\partial \nu} \right|\geq \delta_0,\]
\item $\partial \omega_{f,V_0}$ is locally a graph over $\partial \O$.
\end{enumerate}
Finally, for any $\alpha \in (0;1)$, $\partial \omega_{f,V_0}$ converges to $\partial \O$ in $\mathscr C^{1,\alpha}$ as $V_0\to 1$, uniformly in $f\in \mathcal F(V_0)$.
\end{lemma}

\begin{proof}[Proof of Lemma \ref{Le:V0Lagrange}]
We will proceed in several steps.

\medskip

\noindent \textbf{Convergence to the torsion function.}
This lemma rests upon the study of the torsion function $w_\O$ of $\O$, defined as the solution of
\begin{equation}\label{Eq:Torsion}
\begin{cases}-\Delta w_\O=1&\text{ in }\O\,, 
\\ w_\O=0&\text{ on }\partial \O.\end{cases}\end{equation} By the maximum principle of Hopf there exists $\underline \delta>0$ such that 
\begin{equation}\label{Eq:EstTorsion} \min_{\partial \O}\left|\frac{\partial w_\O}{\partial \nu}\right|\geq \underline \delta.\end{equation}
Furthermore, observe that standard $L^p$-regularity estimates (see for instance \cite[Corollary 9.10]{GilbargTrudinger}) entail that, for any $p\in [1;+\infty)$, there exists a constant $C_p>0$ such that for any $V_0\in (0;1)$ and any $f\in \mathcal F(V_0)$ 
\[ \Vert u_f\Vert_{W^{2,p}(\O)}\leq C_p.\] From Sobolev embeddings, we deduce that for any $\alpha\in (0;1)$ there exists $H_\alpha$ such that for any $V_0\in (0;1)$ and any $f\in \mathcal F(V_0)$ 
\begin{equation}\label{Eq:TorsionHolder} \Vert u_f\Vert_{\mathscr C^{1,\alpha}( \O)}\leq H_\alpha.\end{equation} From these regularity estimates, we obtain at once the uniform convergence result
\begin{equation}\label{Eq:CvTorsion}
\forall \alpha \in (0;1)\,, \lim_{V_0\to 1}\sup_{f\in \mathcal F(V_0)}\Vert u_f-w_\O\Vert_{\mathscr C^{1,\alpha}(\O)}=0.
\end{equation}
%
\textbf{Behaviour of $\mu_{f,V_0}$.}
Let us now prove that 
\begin{equation}\label{Eq:ConvMultiplicateur}\lim_{V_0\to 1}\sup_{f\in \mathcal F(V_0)}\vert \mu_{f,V_0}\vert=0.\end{equation}
Argue by contradiction and assume \eqref{Eq:ConvMultiplicateur} does not hold. 
Then there exist $\gamma>0$, a sequence $\{V_k\}_{k\in \N}$  that satisfies $V_k\underset{k\to \infty}\rightarrow 1$ and, for any $k\in \N$, a function $f_k\in \mathcal F(V_k)$ such that 
\begin{equation}\label{Eq:ConvMultiplicateurContrad}
\forall k\in \N\,, \vert \mu_{f_k,V_0}\vert \geq \gamma>0.
\end{equation}
To obtain a contradiction, we first prove that there exists $\beta_0\,, \e_0'\in (0;1)$ such that
\[ \forall V_0\in (1-\e_0'; 1)\,, \forall f \in \mathcal F(V_0)\,,  \beta_0 w_\Omega \leq u_f\leq w_\O.
\]That $u_f\leq w_\O$ simply follows from the Hopf maximum principle as, for any $V_0\in (0;1)$ and any $f\in \mathcal F(V_0)$ there holds $-\Delta u_f=f\leq1= -\Delta w_\O$. Now fix $\beta_0\in (0;1)$. From \eqref{Eq:CvTorsion} and \eqref{Eq:EstTorsion} there exists $\e_0'>0$ such that for any $V_0\in (1-\e_0';1)$ and any $f\in \mathcal F(V_0)$ we have 
\begin{equation}\label{Eq:OrderNormalDerivative}\beta_0 \max_{\partial \O}\left|\frac{\partial w_\O}{\partial \nu}\right|<\min_{\partial \O}\left|\frac{\partial u_f}{\partial \nu}\right|.
\end{equation} As for any $\alpha \in (0;1)$ we have 
\[ w_\O \in \mathscr C^{1,\alpha}(\O)\,, \sup_{V_0\in (0;1)\,, {f \in \mathcal V_0} }\Vert u_f\Vert_{\mathscr C^{1,\alpha}(\O)}<\infty\] we deduce that there exists $r_0'>0$ small enough such that if we define 
\begin{equation}\label{Eq:Or0} \Omega(r_0'):=\left\{ x \in \O\,, \mathrm{dist}(x,\partial \O)\leq r_0'\right\}\end{equation} then \eqref{Eq:OrderNormalDerivative} along with the Dirichlet boundary conditions implies 
\[\beta_0 w_\O\leq u_f \text{ in }\O(r_0').
\]
From the maximum principle and \eqref{Eq:CvTorsion} we also conclude that if $\e_0'>0$ is small enough  ($r_0'>0$ being fixed) we have 
\[ \beta_0 w_\O<u_f \text{ in }\O\backslash \Omega(r_0').\] Thus we do obtain, if $\e_0'>0$ is small enough that, for any $V_0\in (1-\e_0';1)$, for any $f\in \mathcal F(V_0)$ we have 
\begin{equation}\label{Eq:OrderTorsion}\beta_0 w_\Omega \leq u_f.
\end{equation}
Let us now prove that \eqref{Eq:ConvMultiplicateurContrad} can not hold. From \eqref{Eq:OrderTorsion} we deduce that for any $k\in \N$ we have 
\[ \left\{u_f\geq \mu_{f_k,V_k}\right\}\subset \left\{ w_\O\geq \frac{\mu_{f_k,V_k}}{\beta_0}\right\}\subset \left\{ w_\O\geq \frac{\gamma}{\beta_0}\right\}.\]
Consequently we deduce that 
\[ \forall k\in \N\,, \mathrm{Vol}\left(\left\{u_f\geq \mu_{f_k,V_k}\right\}\right)\leq  \mathrm{Vol}\left(\left\{ w_\O\geq \frac{\gamma}{\beta_0}\right\}\right)\leq 1-s_0
\] for a fixed $s_0\in (0;1)$, since $w_\O>0$ in $\O$ and $\gamma>0$. This is a contradiction with the definition of $\mu_{f_k,V_k}$, whence the conclusion.

\medskip

\noindent \textbf{Study of $\omega_{f,V_0}$.}
From \eqref{Eq:EstTorsion} and the fact that $w_\O \in \mathscr C^{1,\alpha}$ there exists $r_0>0$ such that 
\[ \left|\n w_\O \right|\geq \frac{\underline \delta}2\text{ in }\O(r_0')\] where $\O(r_0)$ is defined in \eqref{Eq:Or0}. Fix such an $r_0>0$.  Since $w_\O>0$ in $\O$, \eqref{Eq:CvTorsion}-\eqref{Eq:ConvMultiplicateur} imply that there exists $\e_0''>0$ such that if $V_0\in (1-\e_0'';1)$ then for any $f\in \mathcal F(V_0)$
\[\{u_f=\mu_{f,V_0}\}\subset \O(r_0).\]
From \eqref{Eq:CvTorsion} we also have, up to reducing $\e_0''>0$,
\[ \inf_{V_0\in (1-\e_0'';1)\,, f \in \mathcal F(V_0)}\left(\inf_{\O(r_0)} \left|\n u_f\right|\right)\geq \frac{\underline\delta}4.\] From the implicit function theorem we deduce first that $\mathrm{Vol}\left(\{u_f=\mu_{f,V_0}\}\right)=0$, second that for any $\alpha \in (0;1)$, $\omega_{f,V_0}$ has a $\mathscr C^{1,\alpha}$ boundary and, finally, that there exists $\delta_0>0$ such that 
\[ \min_{\partial \omega_{f,V_0}} \left| \frac{\partial u_f}{\partial \nu} \right|\geq \delta_0.\]

\noindent \textbf{Convergence of $\partial \omega_{f,V_0}$.}
We first prove that  $\partial \omega_{f,V_0}$ is (locally) a graph over $\partial \O$. Argue by contradiction and assume that there exists a sequence $\{V_k\}_{k\in \N}$ that converges to 1 and, for any $k\in \N$, a function  $f_k\in \mathcal F(V_k)$, a sequence $\{x_k\}_{k\in \N}\in \partial \O^\N$ and two sequences $\{y_{k,0}\}_{k\in \N}$, $\{ y_{k,1}\}_{k\in \N}$ of points of $\partial \omega_{f_k,V_k}$ such that there exist $(t_{k,i})_{k\in \N\,, i=0,1}\in (\R_+\times \{0,1\})^\N$ with 
\[ y_{k,i}=x_k-t_{k,i}\nu(x_k)\,, t_{k,i}\underset{k\to \infty}\rightarrow 0\,, t_{k,0}< t_{k,1}.\]
We consider a closure point $x_\infty$ of $\{x_k\}_{k\in \N}$. By the intermediate value theorem, for any $k\in \N$ there exists $s_k\in (t_{k,0};t_{k,1})$ such that, setting $z_k:=x_k-s_k\nu(x_k)$ we have 
\[ \langle \n u_f(z_k),\nu(x_k)\rangle=0.\] Clearly we have $z_k\underset{k\to \infty}\rightarrow x_\infty$ and $\nu(x_k)\underset{k\to \infty}\rightarrow \nu(x_\infty)$. We obtain from \eqref{Eq:CvTorsion}
\[ \langle \n w_\O(x_\infty)\,, \nu(x_\infty)\rangle=0,\] in contradiction with \eqref{Eq:EstTorsion}. Consequently, as $V_0\to 1$, for any $f\in \mathcal F(V_0)$, $\partial \omega_{f,V_0}$ is a local graph over $\partial \O$; the fact that it converges to $\partial \O$ uniformly in $f\in \mathcal F(V_0)$ is a simple consequence of \eqref{Eq:CvTorsion}.

\end{proof}

We note the following consequence of Lemma \ref{Le:V0Lagrange}: if we denote by $\mathrm{Lip}(\Sigma)$ the Lipschitz-constant of a hypersurface $\Sigma$ then there exists $\e_0>0$ and a constant $M$ such that 
\begin{equation}\label{Eq:EstPerim} \forall V_0\in (1-\e_0;1)\,, \forall f \in \mathcal F(V_0)\,, \mathrm{Lip}(\partial \omega_{f,V_0})+\mathrm{Per}(\omega_{f,V_0})\leq M.\end{equation}
We will use below  \cite[Lemma 3]{Chanillo_2000} to prove that $\omega_{f,V_0}$ actually has an analytic boundary, see Lemma \ref{Le:CriticalPointRegular}; this is not necessary for the time being as we do not need analytic regularity.

We conclude with the following information about $\omega_{f,V_0}$.
\begin{lemma}\label{Le:OmegaConnectedDirichlet}
There exists $\e_0'>0$ such that, for any $V_0\in (1-\e_0';1)$, for any $f\in \mathcal F(V_0)$ the boundary $\partial \omega_{f,V_0}$ is connected.
\end{lemma}
\begin{proof}[Lemma \ref{Le:OmegaConnectedDirichlet}]
We argue by contradiction. Then there exists a sequence $\{V_k\}_{k\in \N}$ converging to 1 and, for any $k\in \N$, a function $f_k\in \mathcal F(V_k)$ such that $\partial \omega_{f_k,V_0}$ has at least two connected components $\Sigma_{k,1}\,, \Sigma_{k,2}$. We pick, for any $k\in \N$, two points $x_{k,i}\in \Sigma_{k,i}$ ($i=1,2$). First of all, for any closure point $y_i$ of the sequence $\{x_{k,i}\}_{k\in \N}$ we have $y_i\in \partial \O$ ($i=1,2$). This is simply due to the fact that $u_{f_k}(x_{k,i})=\mu_{f_k,V_0}\underset{k\to \infty}\rightarrow 0$. Hence, from \eqref{Eq:CvTorsion} we have $w_\Omega(y_i)=0$, and we conclude by the maximum principle applied to the torsion function. Consequently, for $k$ large enough, $\partial \omega_{k,i}\subset \O(r_0')$ where $\O(r_0')$ is defined in \eqref{Eq:Or0} and $r_0'$ is small enough to ensure that $\partial \omega_{f,V_0}$ is a graph over $\partial \O$. From the connectedness of $\partial \O$, we reach the desired contradiction.
\end{proof}
From the proof of this lemma we see that the only thing required is for $\partial \omega_{f,V_0}$ to be a graph over $\partial \O$ whence we may choose $\e_0=\e_0'$ where $\e_0$ is given in Lemma \ref{Le:V0Lagrange}.

\subsection{Applications of the quantitative bathtub principle}
The first step of the proof is to apply the quantitative bathtub principle \cite[Proposition 26]{MRB2020} (see also {\cite{Cianchi2008}}). Before we state this result in the form that will be used let us recall the bathtub principle \cite[Theorem 1.14]{LiebLoss}:  let $u\in \mathscr C^0(\O)$ and let $V_0\in (0;1)$ be such that there exists a unique $\mu_{u,V_0}$ such that 
\begin{equation}\label{Eq:LevelSetu}\mathrm{Vol}(\{u>\mu_{u,V_0}\})=V_0\mathrm{Vol}(\O)=\mathrm{Vol}(\{u\geq \mu_{u,V_0}\}).\end{equation} If there exists a unique $\mu_{u,V_0}$ such that \eqref{Eq:LevelSetu} is satisfied we define 
\begin{equation}\label{Eq:omegau}
\omega_{u,V_0}:=\left\{u>\mu_{u,V_0}\right\}.
\end{equation}
Then:
\begin{equation}\label{Eq:Bathtub}
f_u:=\mathds 1_{\omega_{u,V_0}}\text{ is the unique solution of }\max_{f \in \mathcal F(V_0)}\int_\O f u.\end{equation}
The goal of the quantitative bathtub principle is to quantify the optimality of $f_u$.

\begin{proposition} \cite[Proposition 26]{MRB2020} \label{Pr:QuantitativeBathtub} Let $\alpha \in (0;1)$. For any triplet of constants $M,P\,, \delta_1>0$ there exists a constant $c_1>0$ such that the following holds: for any function $u$ satisfying
\begin{enumerate}
\item there exists a unique $\mu_{u,V_0}$ that satisfies \eqref{Eq:LevelSetu},
\item $\omega_{u,V_0}$ has a $\mathscr C^1$ boundary and 
\[ \inf_{\partial \omega_{u,V_0}}\left|\frac{\partial u}{\partial \nu}\right|\geq \delta_1>0,\]
\item $\Vert u\Vert_{\mathscr C^{1,\alpha}(\O)}\leq M$ and $\mathrm{Per}(\omega_{u,V_0})\leq P,$
\end{enumerate}
then, defining $f_u:=\mathds 1_{\omega_{u,V_0}}$
\begin{equation}\forall f \in \mathcal F(V_0)\,, \int_\O f u\leq \int_\O f_u u-c_1 \Vert f-f_u\Vert_{L^1(\O)}^2.\end{equation}
\end{proposition} 

\paragraph{Application to the thresholding algorithm.}
The goal of this paragraph is to apply Proposition \ref{Pr:QuantitativeBathtub} to the sequence generated by the thresholding algorithm.
\begin{lemma}\label{Le:QuantitativeBathtubThresholding}
Let $\e_0,\delta_0$ be in the conditions of Lemma \ref{Le:V0Lagrange}. For any $V_0\in (1-\e_0;1)$, for any $f_0\in \mathcal F(V_0)$, the sequence $\{f_k\}_{k\in \N}$ is uniquely defined (in the sense that for any $k\in \N$ there holds $\mathrm{Vol}\left(\{u_{f_k}=\mu_{f_k,V_0}\}\right)=0$) and there holds
\begin{equation}
\sum_{k=0}^\infty \Vert f_{k+1}-f_k\Vert_{L^1(\O)}^2<\infty.
\end{equation}
\end{lemma}
\begin{proof}[Proof of Lemma \ref{Le:QuantitativeBathtubThresholding}]
Let $\e_0\,, \delta_0$ in the conditions of Lemma \ref{Le:V0Lagrange}. We fix $V_0\in (1-\e_0;1)$. The fact that for any $k\in \N$ we have 
\[ \mathrm{Vol}\left(\{u_{f_k}=\mu_{f_k,V_0}\}\right)=0\] is a conclusion of Lemma \ref{Le:V0Lagrange}. 

Since $\O$ is $\mathscr C^2$, by elliptic regularity, we know that for any $p\in [1;+\infty)$
\[ \sup_{k\in \N}\Vert u_{f_k}\Vert_{W^{2,p}(\O)}=M_p<\infty\]whence from Sobolev embeddings \cite[Theorem 12.5]{Leoni}, fixing $s\in (0;1)$, we have 
\[ \sup_{k\in \N}\Vert u_{f_k}\Vert_{\mathscr C^{1,s}(\O)}=:M<\infty.\] Furthermore, 
\[ \inf_{k\in \N}\min_{\partial \omega_{u_{f_k},V_0}}\left|\frac{\partial u_{f_k}}{\partial \nu}\right|\geq \delta_0>0.\]Finally, estimate \eqref{Eq:EstPerim} allows us to apply Proposition \ref{Pr:QuantitativeBathtub} to deduce that there exists a constant $c_1>0$ such that
\[ \forall k\in \N\,, \int_\O u_kf_k\leq \int_\O u_k f_{k+1}-c_1\Vert f_{k+1}-f_k\Vert_{L^1(\O)}^2\] or, equivalently, 
\begin{equation}\label{Eq:Credo}\forall k\in \N\,, c_1\Vert f_{k+1}-f_k\Vert_{L^1(\O)}^2\leq \int_\O u_k \left(f_{k+1}-f_k\right).\end{equation} From the concavity of $\mathscr E$ (Lemma \ref{Le:BasicDirichlet})  we also know that
\[ \forall k\in \N\,, \mathscr E(f_{k+1})-\mathscr E(f_k)\leq \dot{\mathscr E}(f_k)[f_{k+1}-f_k]=\int_\O u_k\left(f_k-f_{k+1}\right).\] Thus we obtain
\begin{equation}\label{Eq:Estimatek}\forall k\in \N\,, c_1\Vert f_{k+1}-f_k\Vert_{L^1(\O)}^2\leq \mathscr E(f_k)-\mathscr E(f_{k+1}).\end{equation}
However, there exists $\underline E$ such that 
\begin{equation}\label{Eq:LowerDirichlet}
\forall V_0\in (0;1)\,, \forall f \in \mathcal F(V_0)\,, \underline E\leq \mathscr E(f).\end{equation}
This is a direct consequence of the variational formulation of the Dirichlet energy \eqref{Eq:DirichletEnergy}. Summing the estimates \eqref{Eq:Estimatek} for $k\in \N$ we get 
\[ c_1\sum_{k=0}^\infty \Vert f_{k+1}-f_k\Vert_{L^1(\O)}^2\leq \mathscr E(f_0)-\underline E<\infty.\] As $c_1>0$, the proof is concluded.
\end{proof}

The purpose of this lemma is to give us more insight into the possible asymptotic behaviours of $\{f_k\}_{k\in \N}$. We detail this in the next paragraph.

Let us start with a basic lemma:
\begin{lemma}\label{Le:ClosurePointBgbg} Let $\e_0$ be in the conditions of Lemma \ref{Le:V0Lagrange}. For any $V_0\in (1-\e_0;1)$, for any initialisation $f_0\in \mathcal F(V_0)$, any $L^\infty-*$ closure point $f_\infty$ of the sequence $\{f_k\}_{k\in \N}$ generated by the thresholding algorithm is a bang-bang function:
\begin{equation}\exists E_\infty\subset \O\,, f_\infty=\mathds 1_{E_\infty}.\end{equation} Furthermore $E_\infty=\{u_{f_\infty}>\mu_{u_{f_\infty},V_0}\}$, which is uniquely defined. In other words, $E_\infty$ is a critical set in the sense of Definition \ref{De:CriticalPoint}.

\end{lemma}
\begin{proof}[Proof of Lemma \ref{Le:ClosurePointBgbg}]
We do not relabel the $L^\infty-*$ converging subsequence and write it $\{f_k\}_{k\in \N}$; its closure point is called $f_\infty$. We recall that $\e_0$ is given by Lemma \ref{Le:V0Lagrange}. By elliptic regularity, for any $p\in [1;+\infty)$,
\[ \sup_{k\in \N}\Vert u_{f_k}\Vert_{W^{2,p}(\O)}<\infty.\]
By Sobolev embeddings, for any $\alpha\in (0;1)$,
\[ u_{f_k}\underset{k\to \infty}{\overset{\mathscr C^{1,\alpha}(\O)}\longrightarrow}u_{f_\infty}.\] 
Define, for any $k\in \N$, $\omega_{k}:=\{u_{f_{k}}\geq \mu_{u_{f_k},V_0}\}$. From \eqref{Eq:EstPerim} we deduce that $\{\omega_k\}_{k\in \N}$ satisfies a uniform $\e$-cone property (\cite[Remark 2.4.8]{HenrotPierre}). From \cite[Theorem 2.4.10]{HenrotPierre} the sequence $\{\omega_k\}_{k\in \N}$ converges in the $L^1$-topology to a measurable subset $\omega_\infty$ of $\O$, with $\mathrm{Vol}(\omega_\infty)=V_0\mathrm{Vol}(\O)$; as $f_{k+1}=\mathds 1_{\omega_k}$ we deduce that $f_\infty=\mathds 1_{\omega_\infty}$. Moreover, still from \cite[Theorem 2.4.10]{HenrotPierre} the sequence $\{\partial \omega_k\}_{k\in \N}$ converges to $\partial \omega_\infty$ in the Hausdorff distance. Consequently, passing to the limit in $\partial \omega_k=\{u_{f_k}=\mu_{f_k,V_0}\}$ we deduce that $u_{f_\infty}$ is constant on $\partial \omega_\infty$. $\omega_\infty$ is thus necessarily the superlevel set of $u_{f_\infty}$ with volume $V_0$, which concludes the proof.

\end{proof}

Using this result we can provide a simplification of the behaviour of the sequence $\{f_k\}_{k\in \N}$.
\begin{lemma}\label{Le:MultipleClosurePoints}
Let $\e_0$ be chosen as in Lemma \ref{Le:V0Lagrange}. For any $V_0\in (1-\e_0;1)$, for any initialisation $f_0\in \mathcal F(V_0)$, let $\{f_k\}_{k\in \N}$ be the sequence generated by the thresholding algorithm. Then:
\begin{enumerate}
\item Either $\{f_k\}_{k\in \N}$ has a unique weak $L^\infty-*$ closure point $f_\infty$, in which case $f_k\underset{k\to \infty}\rightarrow f_\infty$ strongly in $L^1(\O)$,
\item Or $\{f_k\}_{k\in \N}$ has an infinite number of closure points.\end{enumerate}
\end{lemma}
\begin{proof}[Proof of Lemma \ref{Le:MultipleClosurePoints}]Observe that if $\{f_k\}_{k\in \N}$ has a weak $L^\infty-*$ closure point $f_\infty$, Lemma \ref{Le:ClosurePointBgbg} ensures that $f_\infty$ is an extreme point of $\mathcal F(V_0)$, whereby  $f_\infty$ is actually a strong $L^1$ closure point of the sequence.

We prove that if the sequence $\{f_k\}_{k\in\N}$ has two distinct closure points $f_{\infty,1}\neq f_{\infty,2}$,  then it has infinitely many closure points. First of all, from Lemma \ref{Le:QuantitativeBathtubThresholding}, 
\begin{equation}\label{Eq:MG}
\Vert f_{k+1}-f_k\Vert_{L^1(\O)}\underset{k\to\infty}\rightarrow 0.
\end{equation}
By Lemma \ref{Le:ClosurePointBgbg} there exist $F_\infty^1\,, F_\infty^2\subset \O$ such that 
\[ f_{\infty,i}=\mathds 1_{F_{\infty,i}}\quad (i=1,2).\]
Let $\delta_1\in (0;\mathrm{dist}_{L^1}(F_{\infty,1},F_{\infty,2})/4)$ and define
\[ \N_1:=\{k\in \N: \min_{i=1,2}(\Vert f_k-f_{\infty,i}\Vert_{L^1(\O)})\geq \delta_1\}.\] From \eqref{Eq:MG}, $\N_1$ is infinite. We denote it as 
\[\N_1=\{i_0,\dots,i_k,\dots\}.\] The sequence $\{f_{i_k}\}_{k\in \N}$ contains an $L^\infty-*$ converging subsequence, still denoted by $\{f_{i_k}\}_{k\in \N}$. Let $f_{\infty,3}$ be its closure point and let us show that 
\[ \min_{i=1,2}(\Vert f_{\infty,3}-f_{\infty,i}\Vert_{L^1(\O)})\geq \delta_1.\]We only show 
\begin{equation}\label{Eq:Kv}
\Vert f_{\infty,3}-f_{\infty,1}\Vert_{L^1(\O)}\geq \delta_1
\end{equation} as the same proof would yield that $\Vert f_{\infty,3}-f_{\infty,2}\Vert_{L^1(\O)}\geq \delta_1$. To prove \eqref{Eq:Kv} we use in a crucial manner that $f_{\infty,1}$ is a bang-bang function (see Definition \ref{De:BangBang}). Consider $h_k:=f_{i_k}-f_{\infty,1}$. Then we have 

\[h_k=h_{k,+}\mathds 1_{F_{\infty,1}^c}-h_{k,-}\mathds 1_{F_{\infty,1}}
\] with $h_{k,\pm}\geq 0$. This is a simple consequence of the fact that $f_{\infty,1}$ is bang-bang. Furthermore, we have 
\[\int_{F_{\infty,1}^c} h_{k,+}+\int_{F_{\infty,1}}h_{k,-}\geq \delta.
\] 
We consider a weak closure point $h_{\infty,+}$ of $\{h_{k,+}\mathds 1_{F_{\infty,1}^c}\}_{k\in \N}$ and a weak closure point $h_{\infty,-}$ of $\{h_{k,-}\mathds 1_{F_{\infty,1}}\}_{k\in \N}$. We obtain
\[ \int_{F_{\infty,1}^c} h_{\infty,+}+\int_{F_{\infty,1}}h_{\infty,-}\geq \delta.\]
 Furthermore, by linearity of the weak convergence we have
\[ f_{\infty,3}-f_{\infty,1}=h_{\infty,+}\mathds 1_{F_{\infty,1}^c}-h_{\infty,-}\mathds 1_{F_{\infty,1}}.\] Consequently
\[ \Vert f_{\infty,3}-f_{\infty,1}\Vert_{L^1(\O)}\geq \delta,\] whence the conclusion. Thus we deduce that \eqref{Eq:Kv} holds and, adapting the reasoning for $f_{\infty,2}$, we obtain that 
\[ \min_{i=1,2}\left(\Vert f_{\infty,3}-f_{\infty,i}\Vert_{L^1(\O)}\right)\geq \delta.\] By Lemma \ref{Le:ClosurePointBgbg}, $f_{\infty,3}$ is a bang-bang function, and $\{f_k\}_{k\in \N}$ has three distinct closure points. It then suffices to iterate the procedure to construct  an infinite sequence of closure points. This concludes the proof.\end{proof}

\paragraph{The goal of the next sections.} To obtain the uniqueness of the closure point and the strong convergence to a local minimiser, we need to show that every critical point (in the sense of Definition \ref{De:CriticalPoint}) is actually a local minimiser and that these critical points are isolated. This is done using shape derivatives and diagonalisation of shape hessians.

\subsection{Qualitative study of critical points}
In this section we give an in-depth analysis of critical points (in the sense of Definition \ref{De:CriticalPoint}). We begin with the analyticity of the boundaries of critical sets.
\paragraph{Regularity of critical sets.} We will be using a shape derivative formalism and compute second order shape derivatives at critical shapes in order to conclude as to their minimality and to obtain a full convergence result for the thresholding algorithm. Doing so requires some regularity (at least $\mathscr C^3$ for second-order shape derivatives) of the boundary of critical sets. It should be noted that usually this type of regularity is proved for minimal sets (\emph{i.e.} for $E\subset \O$ such that $\mathds 1_E$ is a solution of \eqref{Pv:Dirichlet}). A paradigmatic result is the regularity of minimal sets in two dimensions \cite{Chanillo2008}. However, the hard part in proving this regularity is usually obtaining an estimate of the gradient of the state function on the boundary of the optimal set. Working with a large volume constraint allows to bypass this regularity problem, as gradient estimates are readily provided by Lemma \ref{Le:V0Lagrange}, and thus enable us to apply \cite[Lemma 3, Theorem 8]{Chanillo_2000}.
\begin{lemma}\label{Le:CriticalPointRegular}
Let $\e_0>0$ be as in Lemma \ref{Le:V0Lagrange}. Any critical set $E$ (in the sense of Definition \ref{De:CriticalPoint}) has a compact analytic boundary and, furthermore, is uniformly bounded in the $\mathscr C^2$ topology.
\end{lemma}
\begin{proof}[Proof of Lemma \ref{Le:CriticalPointRegular}] This proof is an adaptation of \cite[Proof of Theorem 8]{Chanillo_2000}. Let us first recall the following simpler version of \cite[Lemma 3]{Chanillo_2000}:
\begin{lemma}[Lemma 3, \cite{Chanillo_2000}]\label{Le:Chanillo}
Let $h:\R\to \R$ be a locally bounded function and $w\in \mathscr C^1(\O)$ be a solution of 
\[\Delta w=h(w)\text{ in }\O.\]Let $x_0\in \O$ be such that $\n w(x_0)\neq 0$. There exists a ball $\mathbb B(x_0;r)$ ($r>0$) such that the set $\{x\in \B:\, w(x)=w(x_0)\}$ is an analytic hypersurface of $\R^d$.
\end{lemma}
The original version of \cite[Lemma 3]{Chanillo_2000} also features a dependency of $h$ on $\n w$; this is not necessary here. Now consider a critical set $E\subset \O$. Let $u_{\mathds 1_E}$ be the solution of \eqref{Eq:Main} with $f=\mathds 1_E$ and let $\mu_E$ be such that 
\begin{equation}\label{Eq:Ksv}
E=\{u_{\mathds 1_E}>\mu_E\}.
\end{equation} The fact that \eqref{Eq:Ksv} holds follows from Lemma \ref{Le:V0Lagrange}.  Defining the (locally bounded) function $h:\eta\mapsto \mathds 1_{(\mu_E;+\infty)}(\eta)$, we thus have 
\[-\Delta u_{\mathds 1_E}=h(u_{\mathds 1_E}).\]
From Lemma \ref{Le:V0Lagrange}, for any $x_0\in \partial E=\{x:\, u_{\mathds 1_E}(x)=\mu_E\}$, we have 
\[ \n u_{\mathds 1_E}(x_0)\neq 0.\] Applying Lemma \ref{Le:Chanillo} we deduce that $\partial E$ is locally analytic.  In other words, we may write 
\[ \partial E=\underset{x\in \partial E}\bigcup (\B(x;r_x)\cap \partial E)\] with $r_x>0$ for any $x\in \partial E$ and $\B(x;r_x)\cap \partial E$ is an analytic hypersurface of $\R^n$. As $\partial E$ is compact by continuity of $u_{\mathds 1_E}$, we may extract a finite covering of $\partial E$ and hence conclude that $\partial E$ is a compact analytic hypersurface. To obtain the uniform $\mathscr C^2$ bounds it suffices to conclude as in \cite[Lemma 3, Theorem 8]{Chanillo_2000}, by using \eqref{Eq:EstPerim}\footnote{In the proof of \cite[Lemma 3]{Chanillo_2000} the analytic bounds obtained only depend on the regularity of the  local change of coordinates, which here is uniform in $E$ from Lemma \ref{Le:V0Lagrange} and the implicit function theorem.}.
\end{proof}

\paragraph{Preliminary considerations about shape derivatives at critical shapes.} We first identify the functional $\mathscr E$ with a shape functional $\mathscr F$, by defining
\[ \mathscr F: \O\supset E\mapsto \mathscr E(\mathds 1_E).\]
For any compact $\mathscr C^3$ subset $E$ of $\O$ and for any compactly supported vector field $\Phi\in W^{2,\infty}(\O;\R^d)$,  we can define, for any $t\in (-1;1)$ small enough, the function 
\[ e_{E,\Phi}(t):=\mathscr F\Big((\mathrm{Id}+t\Phi)E\Big).\] Provided they exists, the first (resp. second) order derivative of $\mathscr F$ at $E$ in the direction $\Phi$ is defined as 
\[ \mathscr F'(E)[\Phi]=e_{E,\Phi}'(0)\,, \text{resp. }\mathscr F''(E)[\Phi,\Phi]=e_{E,\Phi}''(0).
\]
We first check that $\mathscr E$ is indeed twice differentiable in the sense of shapes.
\begin{lemma}\label{Le:Differentiability}
For any $\mathscr C^3$ subset $E$ of $\O$, $\mathscr F$ is twice differentiable at $E$ in the direction of any compactly supported vector-field $\Phi\in W^{2,\infty}(\O;\R^d)$.
\end{lemma}
As this lemma is proved by a direct adaptation of \cite{MignotMuratPuel} we omit it here and focus in a subsequent paragraph on the computation of first and second order shape derivatives.

\paragraph{First order shape derivative and definition of the Lagrange multiplier.}
In terms of optimality conditions, since we are working with volume constraints, we  have to work with the  shape derivative of the volume functional. We recall \cite[Section 5.9.3, first example]{HenrotPierre} that the map
\[ \O\supset E\mapsto \mathrm{Vol}(E)\] is twice shape differentiable (\emph{i.e.} differentiable in the direction of any compactly supported vector field $\Phi \in W^{2,\infty}(\O)$) at any $\mathscr C^3$ domain $E$ and that, for any compactly supported vector field $\Phi$, there holds 
\begin{equation}\label{Lam} \mathrm{Vol}'(E)[\Phi]=\int_{\partial E}\langle \Phi,\nu_E\rangle\,, \mathrm{Vol}''(E)[\Phi,\Phi]=\int_{\partial E}\mathscr H_E \langle \Phi,\nu\rangle^2\end{equation} where $\nu_E$ is the normal vector on $\partial E$ and $\mathscr H_E$ is the mean curvature of $\partial E$.

For semantical convenience, we introduce the following definition:
\begin{definition}[Critical shape]\label{De:CriticalShape}
A $\mathscr C^2$ shape $E$ is \textbf{ a critical shape } for $\mathscr F$ if, for any compactly supported $\Phi\in W^{2,\infty}(\O)$, 
\[ \int_{\partial E} \langle \Phi,\nu_E\rangle=0\Rightarrow \mathscr F'(E)[\Phi]=0.\]
\end{definition}
In order to exploit second order optimality conditions it is more convenient to use a Lagrange multiplier. If a shape $E$ is critical, then the condition given in Definition \ref{De:CriticalShape} rewrites as: there exists a constant $\mu_E$ such that, for any compactly supported vector field $\Phi\in W^{2,\infty}(\O;\R^d)$, 
\begin{equation}\label{Eq:VolDeriv}\left(\mathscr F+\mu_E\mathrm{Vol}\right)'(E)[\Phi]=0.\end{equation} Computing this Lagrange multiplier is an important step in our subsequent analysis; to do it we need an expression of the first order derivative.
\begin{lemma}\label{Le:FirstSD}
For any $\mathscr C^3$ shape, for any compactly supported vector field $\Phi \in W^{2,\infty}(\O;\R^d)$, 
\[ \mathscr F'(E)[\Phi]=-\int_{\partial E} u_{\mathds 1_E}\langle \Phi,\nu_E\rangle.\] Furthermore, the shape derivative of the map $E\mapsto u_{\mathds 1_E}$ at $E$ in the direction $\Phi$ is the unique solution of 
\begin{equation}\label{Eq:FSDu}
\begin{cases}
-\Delta u_\Phi'=0&\text{ in }\O\,, 
\\ \left\llbracket \partial_\nu u_\Phi' \right\rrbracket=-\langle \Phi,\nu_E\rangle&\text{ on }\partial E
\end{cases}
\end{equation} where $\llbracket\cdot\rrbracket$ denotes the jump of a function across a hypersurface.

\end{lemma} 
\begin{proof}[Proof of Lemma \ref{Le:FirstSD}]
To prove this lemma we first need to compute the shape derivative of the map $u\mapsto u_{\mathds 1_E}$. These computations were already carried out (for the shape derivative of the first eigenvalue of the operator $-\Delta-\mathds 1_E$) in details in \cite{MRB2020}, but we sketch them in the present case for the sake of completeness. Fix $E$ and a compactly supported vector field $\Phi$. Define, for any $t$ such that $|t|$ small enough, $E_t:=\mathds 1_{(\mathrm{Id}+t\Phi)E}$ and $u_t$ as the solution of \eqref{Eq:Main} with $f=\mathds 1_{E_t}$. The weak formulation of the equation on $u_t$ reads:
\begin{equation}\forall v\in W^{1,2}_0(\O)\,, \int_\O \langle \n u_t,\n v\rangle=\int_{E_t}v.\end{equation}
Taking the derivative of this formulation at $t=0$, it appears that the shape derivative $u_\Phi'$ satisfies the equation
\begin{equation}\label{Eq:WeakFSD}
\forall v \in W^{1,2}_0(\O)\,, \int_\O \langle \n u_\Phi',\n v\rangle=\int_{\partial E} v\langle \Phi,\nu_E\rangle.
\end{equation}
Alternatively, this rewrites as: $u_\Phi'$  is the unique solution of the equation 
\[\begin{cases}
-\Delta u_\Phi'=0&\text{ in }\O\,, 
\\ \left\llbracket \partial_\nu u_\Phi' \right\rrbracket=-\langle \Phi,\nu_E\rangle&\text{ on }\partial E
\end{cases}
\] exactly as claimed in the statement of the theorem. Now, going back to the definition of $\mathscr F$ we obtain 
\[ \mathscr F'(E)[\Phi]=\int_\O \langle \n u_{\mathds 1_E},\n u_\Phi'\rangle-\int_E u_\Phi'-\int_{\partial E}u_{\mathds 1_E}\langle \Phi,\nu_E\rangle.\]Using the weak formulation of \eqref{Eq:Main} this gives
\[ \mathscr F'(E)[\Phi]=-\int_{\partial E}u_{\mathds 1_E}\langle \Phi,\nu_E\rangle.\]
\end{proof}
This expression allows us to prove the following result
\begin{lemma}\label{Le:EquivalenceCritical}
A shape $E\subset \O$ is a critical set in the sense of Definition \ref{De:CriticalPoint} if, and only if it is a critical shape in the sense of Definition \ref{De:CriticalShape}.
\end{lemma}
\begin{proof}[Proof of Lemma \ref{Le:EquivalenceCritical}]
Let us first notice that from Lemma \ref{Le:CriticalPointRegular}, critical shapes in the sense of Definition \ref{De:CriticalPoint} are analytic and thus in particular $\mathscr C^2$. 

Consider now a critical point $E$ in the sense of Definition \ref{De:CriticalPoint}. By definition there exists a constant $\mu_E$ such that $u_{\mathds 1_E}=\mu_E$ on $\partial E$. As $E$ is $\mathscr C^2$ from Lemma \ref{Le:CriticalPointRegular} we can legitimately compute the first order shape derivative of $\mathscr F$ at $E$ in the direction of a fixed compactly supported vector field $\Phi$. By Lemma \ref{Le:FirstSD}
\[ \mathscr F'(E)[\Phi]=-\int_{\partial E}u_{\mathds 1_E}\langle \Phi,\nu_E\rangle=-\mu_E\int_{\partial E}\langle \Phi,\nu_E\rangle.\] Hence, if $\int_{\partial E}\langle \Phi,\nu_E\rangle=0$ we have $\mathscr F'(E)[\Phi]=0$ and thus $E$ is a critical shape in the sense of Definition \ref{De:CriticalShape}.

Conversely, consider a shape $E$ of class $\mathscr C^3$ that is critical in the sense of Definition \ref{De:CriticalShape}. As for any compactly supported vector field $\Phi$ such that $\int_{\partial E}\langle \Phi,\nu_E\rangle=0$ we must have 
\[ \int_{\partial E}u_{\mathds 1_E}\langle \Phi,\nu_E\rangle=0\] the fundamental lemma of the calculus of variations implies the existence of a constant $\mu_E>0$ such that 
\[u_{\mathds 1_E}=\mu_E\text{ on }\partial E.\] From the maximum principle applied to \eqref{Eq:Main} we already know that $u_{\mathds 1_E}>0$ in $\O$, hence $\mu_E>0$. Consequently $\partial E\cap \partial \O=\emptyset$ and $u_{\mathds 1_E}$ is a solution of 
\begin{equation}
\begin{cases}-\Delta u_{\mathds 1_E}=1&\text{ in }E\,, 
\\ u_E=\mu_E&\text{ on }\partial E\end{cases}\text{ and } \begin{cases}-\Delta u_{\mathds 1_E}=0&\text{ in }\O\backslash \overline E\,, \\u_{\mathds 1_E}=\mu_E&\text{ on }\partial E\,, \\ u_{\mathds 1_E}=0&\text{ on }\partial \O.\end{cases}\end{equation} As $\partial E\cap \partial \O=\emptyset$, both $E$ and $\O\backslash \overline E$ are smooth open sets. From the strong maximum principle we deduce that $u_{\mathds 1_E}>\mu_E$ in $E$ and that $u_{\mathds 1_E}<\mu_E$ in $\O\backslash \overline E$. Thus $E$ is a critical set in the sense of Definition \ref{De:CriticalPoint}.
\end{proof}
From this analysis we deduce the following:
\begin{lemma}\label{Le:LagrangeMultiplier}
Assume $E$ is a critical shape in the sense of Definition \ref{De:CriticalShape}. The associated Lagrange multiplier is $\mu_E$ where $E=\{u_{\mathds 1_E}>\mu_E\}$ in the sense that 
\[ \forall \Phi\in W^{2,\infty}(\O;\R^d)\,, \Phi \text{ compactly supported }\,, \left(\mathscr F+\mu_E\mathrm{Vol}\right)'(E)[\Phi]=0.\]
\end{lemma}

To alleviate notations we introduce a notation for the shape Lagrangian:
\begin{definition}[Shape Lagrangian]\label{De:ShapeLagrangian}Let $E$ be a critical shape for $\mathscr F$. The \textbf{associated shape Lagrangian} is 
\[ \mathscr L_E:=\mathscr F+\mu_E \mathrm{Vol}\] where $E=\{u_{\mathds 1_E}>\mu_E\}$.\end{definition}

With the above ingredients at hand, we now move to second-order shape derivatives.

\paragraph{First computations and elements for shape hessians.} In the subsequent paragraphs we use indifferently the expressions "shape hessians" and "second order shape derivatives". We begin with an expression of the shape hessian at a critical shape $E$ in the direction $\Phi$.
\begin{lemma}\label{Le:SecondSD}
Let $E$ be a critical shape in the sense of Definition \ref{De:CriticalShape} and let $\Phi\in W^{2,\infty}(\O;\R^d)$ be a compactly supported vector field. The second-order shape derivative of the shape Lagrangian $\mathscr L_E$ (defined in Definition \ref{De:ShapeLagrangian}) is given by the expression
\[\mathscr L_E''(E)[\Phi,\Phi]= -\int_{\partial E}u_\Phi'\langle \Phi,\nu_E\rangle-\int_{\partial E}\frac{\partial u_E}{\partial \nu}\langle \Phi,\nu_E\rangle^2
\]
where $u_\Phi'$ solves \eqref{Eq:FSDu}.
\end{lemma}
\begin{proof}[Proof of Lemma \ref{Le:SecondSD}]First of all, since $E$ is a critical shape, it suffices to work with vector fields that are normal to $\partial E$ (see \cite[Remark on page 246]{HenrotPierre}).
We recall that the formula of Hadamard gives
\[
\left.\frac{d}{dt}\right|_{t=0}\int_{\Gamma(t)}f(t)=\int_{\Gamma(0)} f'(0)+\int_{\Gamma(0)}\left(\mathscr H_{\Gamma(0)} f+\frac{\partial f}{\partial \nu}\right)\langle \Phi,\nu\rangle\] where 
$\mathscr H_{\cdot}$ denotes the mean curvature of a hypersurface.

Let us start from the fact that, at a given shape $E$ we have (Lemma \ref{Le:FirstSD})
\[\mathscr F'(E)[\Phi]=-\int_{\partial E} u_{\mathds 1_E} \langle \Phi,\nu_E\rangle.\] Applying the formula of Hadamard we obtain
\[
\mathscr F''(E)[\Phi,\Phi]=-\int_{\partial E} u_\Phi'\langle \Phi,\nu\rangle-\int_{\partial E}\left( \mathscr H_{ E} u_{\mathds 1E}+\frac{\partial u_{\mathds 1_E}}{\partial \nu}\right)\langle \Phi,\nu_E\rangle^2.
\]
Taking into account the definition of the Lagrange multiplier (Lemma \ref{Le:LagrangeMultiplier}) and \eqref{Lam}
 we finally derive the following expression for the second-order shape derivative of the Lagrangian (see Definition \ref{De:ShapeLagrangian})
\begin{equation}\label{Eq:EllDirichletLagrangianDS}
\mathscr L_{E}''(E)[\Phi,\Phi]=-\int_{\partial E} u_\Phi'\langle \Phi,\nu\rangle-\int_{\partial E} \frac{\partial u_{\mathds 1_E}}{\partial \nu}\langle \Phi,\nu_E\rangle^2.\end{equation}
\end{proof}

\paragraph{The question of coercivity of shape Lagrangians.} The goal of upcoming sections is to provide coercivity results on shape Lagrangians at critical shapes. By coercivity we mean not only that the shape hessian is positive, \emph{i.e.} that for any compactly supported vector field $\Phi\in W^{2,\infty}(\O;\R^d)$ we have 
\[ \mathscr L_E''(E)[\Phi,\Phi]>0\text{ whenever }\langle \Phi,\nu_E\rangle\neq 0\text{ on }\partial E\] but  that we can even quantify this positivity, in the sense that there exists a constant $c_E>0$ (independent of $\Phi$) such that 
\[ \mathscr L_E''(E)[\Phi,\Phi]\geq c_E \Vert \langle \Phi,\nu_E\rangle \Vert_{L^2(\partial E)}^2.\]
This leads to introducing the following definition:
\begin{definition}[$L^2$-stability]\label{De:StableShape}
A critical shape $E$ is said to be $L^2$-\textbf{stable} if there exists a constant $c_E>0$ such that
 for any compactly supported vector field $\Phi\in W^{2,\infty}(\O;\R^d)$, 
\[ \mathscr L_E''(E)[\Phi,\Phi]\geq  c_E\Vert \langle \Phi,\nu_E\rangle\Vert_{L^2(\partial E)}^2.\]
\end{definition}

It should be noted that we can not expect a better coercivity norm than $L^2(\partial E)$; when diagonalising the shape hessian we actually will prove the optimality of this norm. While weaker than the usual $H^{\frac12}(\partial E)$ norm \cite{DambrineLamboley,DambrinePierre,henrotpierrerihani} this is still enough to obtain local quantitative inequalities \cite{MRB2020}. 

Our goal is henceforth to prove the following coercivity result:
\begin{proposition}\label{Pr:CriticalCoercivity}There exists $\e_1>0\,, \e_1\leq \e_0\,, \underline c>0$ such that, for any $V_0\in (1-\e_1;1)$, for any critical set $E$ (in the sense of Definition \ref{De:CriticalShape}), for any compactly supported vector field $\Phi\in W^{2,\infty}(\O;\R^d)$, 
\[ \mathscr L_E''(E)[\Phi,\Phi]\geq \underline c\Vert \langle \Phi,\nu\rangle\Vert_{L^2(\partial E)}^2.\]
\end{proposition}
Before we prove this proposition we need to introduce the diagonalisation basis.
We consider a fixed critical shape $E\subset \O$ and work with $V_0\in (1-\e_0;1)$, where $\e_0$ is given by Lemma \ref{Le:V0Lagrange}. Heuristically, given that $u_\Phi'$ satisfies \eqref{Eq:FSDu}, it is natural to solve the following eigenvalue problem:  find $(\lambda,\phi)\in \R\times W^{1,2}(\O)$ such that
 \begin{equation}\label{Eq:EllDirichletSteklov}
 \begin{cases}
 -\Delta \phi=0&\text{ in }E\cup \left(E\right)^c
 \\ \llbracket \partial_\nu \phi\rrbracket_{E}=-\lambda \rho \phi&\text{ on } \partial E\,, 
 \\ \int_{\partial E} \phi^2>0\,, 
 \\ \phi \in W^{1,2}_0(\O)
 \end{cases}
 \end{equation}
where $\rho$ is a weight to be determined. Let us simply consider the eigenvalue problem \eqref{Eq:EllDirichletSteklov} with a weight $\rho \in L^\infty(\partial E;[\delta,1/\delta])$ for some small $\delta>0$. To study this spectral problem we introduce the weighted space
$$L^2_{\rho}(\partial E):=\left\{f\in L^2(\partial E), \int_{\partial E}\rho f^2<\infty\right\}=L^2(\partial E).$$ The last equality comes from the fact that the weight $\rho$ is uniformly bounded. The only difference here is of course the scalar product; we work with 
\[ \langle\cdot,\cdot\rangle_\rho:(f,g)\mapsto \int_{\partial E}\rho f g.\]

Under these assumptions it is fairly standard to obtain the existence of a spectral basis associated with \eqref{Eq:EllDirichletSteklov}. Define the operator $T_\rho:L^2_\rho(\partial E)\to L^2(\partial E)$ as follows: for any $f\in L^2_\rho(\partial E)$  let $v_f\in W^{1,2}(\O)$ be the unique solution to 
\begin{equation}\label{Eq:Aux}
\begin{cases}
-\Delta v_f=0	&\text{ in }\O\,, 
\\ \llbracket\partial_\nu v_f\rrbracket=-\rho f&\text{  on }\partial E\,, 
\\ v_f\in W^{1,2}_0(\O).\end{cases}\end{equation} 
The fact that for any $f\in L^2_\rho(\partial E)$ the function $v_f$ exists and is unique simply follows by minimising the functional
\[ W_{\rho,f}:W^{1,2}_0(\O)\ni v\mapsto \frac12\int_\O|\n v|^2-\int_{\partial E} \rho v f.\] 

Let $\mathrm{Tr}_{E}:W^{1,2}(\O)\to L^2(\partial E)$ be the trace operator. As $E$ is analytic (Lemma \ref{Le:CriticalPointRegular}) this operator is well-defined. Set 
\begin{equation}T_\rho(f):=\mathrm{Tr}_{E}(v_f).\end{equation} 
In particular, \[\forall f \in L^2_\rho(\partial E)\,, T_\rho(f)\in L^2_\rho(\partial E).\]
Clearly, $T_\rho$ is symmetric, positive and compact.
%
In particular, by the spectral decomposition theorem, there exists a sequence of eigen-elements $\left\{\left(\sigma_{k,\rho}\,, \psi_{k,\rho}\right)\right\}_{k\in \N}\in \left(\R_+^*\times L^2_\rho(\partial E)\right)^\N$ such that 
$$\sigma_{k,\rho}\text{ is non-increasing in $k$ and }\sigma_{k,\rho}\underset{k\to \infty}\longrightarrow 0,$$

$$\{\psi_{k,\rho}\}_{k\in \N}\text{ is an orthonormal basis of $L^2_\rho(\partial E)$}$$ and 
$$\forall k\in \N\,, \quad T_\rho(\psi_{k,\rho})=\sigma_{k,\rho}\psi_{k,\rho}.$$

We set $\phi_{k,\rho}:=v_{\psi_{k,\rho}}$ to extend it to a function on $\O$. Hence,
\begin{equation}\label{Eq:eigen}
\begin{cases}
-\Delta \phi_{k,\rho}=0&	\text{ in }\O\,, 
\\ \llbracket\partial_\nu \phi_{k,\rho}\rrbracket=-\rho \psi_{k,\rho}=-\frac1{\sigma_{k,\rho}}\rho \phi_{k,\rho}&\text{  on }\partial E\,, 
\\  \phi_{k,\rho}=0&\text{ on }\partial\O.\end{cases}\end{equation} We thus obtain, defining 
$$\lambda_{k,\rho}:=\frac1{\sigma_{k,\rho}}$$ the system
\begin{multline}\label{Eq:eigen2}
\forall k\in \N^*\,, \begin{cases}
-\Delta \phi_{k,\rho}=0&	\text{ in }\O\,, 
\\ \llbracket\partial_\nu \phi_{k,\rho}\rrbracket=-\lambda_{k,\rho}\rho \phi_{k,\rho}&\text{  on }\partial E\,, 
\\ \phi_{k,\rho}=0&\text{ on }\partial \O\end{cases} \\\text{ and } \forall k\,, k'\in \N^*\,, \int_{\partial E}\rho \phi_{k,\rho}\phi_{k',\rho}=\delta_{k,k'}.\end{multline} The family $\{\phi_{0,\rho}\}\cup \{\phi_{k,\rho}\}_{k\in \N^*}=\{\phi_{k,\rho}\}_{k\in \N}$ is an orthonormal basis of $L^2_\rho(\partial E)$ for the scalar product $ \langle \cdot,\cdot \rangle_\rho$.

From the Courant-Fisher principle we have, furthermore, the following characterisation of the first eigenvalue:
\begin{equation}\label{Eq:RayleighRho}
\lambda_{0,\rho}=\min_{u\in W^{1,2}_0(\O)\,, \int_{\partial E}u^2>0}\frac{\int_\O|\n u|^2}{\int_{\partial E}\rho u^2}.
\end{equation}

\paragraph{Diagonalisation of $\mathscr L_{E}''(E)$.}
We can now turn back to the expression of the second-order shape derivative of the Lagragian \eqref{Eq:EllDirichletLagrangianDS}. We want to determine the weight $\rho$. We use the following notational convention: for any $\Phi \in W^{2,\infty}(\O;\R^d)$,
 \[\p:=\langle \Phi\,, \nu\rangle\,, \tilde\p:=\frac1{\rho} \p.\] The first-order shape derivative $u_\Phi'$ satisfies
\begin{equation}\label{Eq:Chop}
\begin{cases}
-\Delta u_\Phi'=0&\text{ in }\O\,, 
\\\left\llbracket\frac{\partial u_\Phi'}{\partial \nu}\right\rrbracket=-\rho \tilde \p&\text{ on }\partial E\,, 
\\ u_\Phi'=0&\text{ on }\partial \O.\end{cases}\end{equation} 
%
%

 We decompose $\tilde \p$ in the basis $\{\phi_{k,\rho}\}_{k\in \N}$ as
\begin{equation}
\tilde\p=\sum_{k=0}^\infty \alpha_{k,\rho} \phi_{k,\rho}.
\end{equation}
Thus, 
$$u_\Phi'=\sum_{k=0}^\infty\frac{\alpha_{k,\rho}}{\lambda_{k,\rho}} \phi_{k,\rho}.$$
As a conclusion
\begin{align*}
\mathscr L_{E}''(E)[\Phi,\Phi]&=-\int_{\partial E} u_\Phi'\langle \Phi\,, \nu_E\rangle
-\int_{\partial E}\partial_\nu u_{\mathds 1_E}\langle \Phi, \nu_E\rangle^2
\\&=-\int_{\partial E}\rho \tilde\p u_\Phi'
-\int_{\partial E}\rho^2\partial_\nu u_{\mathds 1_E}\left(\tilde \p\right)^2.
\end{align*}
Let us choose 
\[\rho_E=-\frac1{\partial_\nu u_{\mathds 1_E}}.\] From Lemma \ref{Le:V0Lagrange}, we have $\rho_E\in L^\infty([\delta;1/\delta])$ for $\delta$ small enough. Finally, this yields

 \begin{align*}
\mathscr L_{E}''(E)[\Phi,\Phi]&=-\int_{\partial E} u_\Phi'\langle \Phi\,, \nu_E\rangle
-\int_{\partial E}\partial_\nu u_{\mathds 1_E}\langle \Phi, \nu_E\rangle^2
\\&=-\int_{\partial E}\rho_E \tilde\p u_\Phi'
-\int_{\partial E}\rho_E^2\partial_\nu u_{\mathds 1_E}\left(\tilde \p\right)^2
\\&=\sum_{k=0}^\infty \alpha_{k,\rho_E}^2 \left(1-\frac1{\lambda_{k,\rho_E}}\right).
\end{align*}
In particular, given that the sequence $\{\lambda_{k,\rho_E}\}_{k\in \N}$ is non-decreasing, we have the estimate from below:
\begin{equation}\label{Eq:LowerEstimateLagrangian} \mathscr L_E''(E)[\Phi,\Phi]\geq  \left(1-\frac1{\lambda_{0,\rho_E}}\right)\sum_{k=0}^\infty\alpha_{k,\rho_E}^2= \left(1-\frac1{\lambda_{0,\rho_E}}\right)\Vert \langle \Phi,\nu_E\rangle\Vert_{L^2(\partial E)}^2.\end{equation}
Hence we have the following sufficient condition for the stability of a critical shape:
\begin{lemma}\label{Le:SufficientStable}
If $\lambda_{0,\rho_E}>1$, then $E$ is stable in the sense of Definition \ref{De:StableShape}. 
\end{lemma}

To prove Proposition \ref{Pr:CriticalCoercivity} we study the asymptotic behaviour of $\lambda_{0,\rho_E}$ as $V_0\to 1$.

\paragraph{Asymptotic behaviour of $\lambda_{0,\rho_E}$ as $V_0\to 1$.}
We prove here the following lemma:
\begin{lemma}\label{Le:EigenvalueBlowup}
There holds \[\min_{E\text{ critical shape for the volume constraint $V_0$}}\lambda_{0,E}\underset{V_0\to 1}\rightarrow \infty.\]
\end{lemma}
\begin{remark}[Heuristic approach to Lemma \ref{Le:EigenvalueBlowup}]Assume $\O=(-1;1)$ and $V_0=1-\e$, where $\e>0$ is a small parameter. In this case, we may use the Schwarz rearrangement to obtain that the unique solution of \eqref{Pv:Dirichlet} is $f^*=\mathds1_{I_\e}$ where 
\[ I_\e=(-1+\e/2;1-\e/2).\] Let $u_\e$ be the associated solution of \eqref{Eq:Main}. Clearly, 
\[u_\e\underset{\e \to 0}\rightarrow u_0\] in $\mathscr C^1(\O)$, where $u_0$ is the solution of \eqref{Eq:Main} with $f\equiv 1$. In particular, 
\[ |u_\e'(\pm 1\mp\e/2)|=a_\e\underset{\e\to 0}\rightarrow1.\]
Now let us study the lowest eigenvalue $\lambda_{0,\rho_{I_\e}}$ which, for the sake of notational convenience, we rewrite as $\lambda_0(\e)$. But notice the following thing: if we define 
\[ \p_{0,\e}:x\mapsto \begin{cases}1\text{ if }|x|\leq 1-\e/2,
\\ \frac{2}{\e a_\e} (1-|x|) \text{ if } 1-\e/2\leq |x|\leq 1
\end{cases}\]
it appears that $\p_{0,\e}$ is an eigenfunction associated with the eigenvalue $\frac2\e$. As it has a constant sign, and as the eigenfunction associated with the lowest eigenvalue is the only one that has a constant sign, we deduce that 
\[ \lambda_{0,\rho_{I_\e}}(\e)=\frac{2}\e,\]
which indeed goes to $\infty$ as $\e\to 0^+$.

The situation is of course more complicated in higher dimensions. Here is a roadmap to prove the same type of results. We already observed in the proof of Lemma \ref{Le:V0Lagrange} that $u_f\underset{V_0\to 1}\rightarrow w$ uniformly in $f$, as $\e\to 0$, where $w$ is the torsion function of the set $\O$. As $\O$ has a regular boundary, the Hopf maximum principle entails that 
\[ \inf_{\partial \O}\left(-\frac{\partial w}{\partial \nu}\right)=\delta(\O)>0.\]
Now consider the set $X^*_\e$ of critical shapes. From the definition of criticality, it follows that, for any $E^*\in X^*$, there exists $\mu(E^*)$ such that $E^*=\{u_{\mathds 1_{E^*}}>\mu(E^*)\}$. 
We want to prove that 
\[ \lim\inf_{\e \to 0^+}\min_{E^*\in X^*_\e}\lambda_{0,\rho_{E^*}}=+\infty.\] We argue by contradiction and assume this is not the case. This gives a sequence $\e\to 0^+$ and a sequence $\{E_\e\}_{\e>0}\in \prod_{\e>0}X^*_\e$ such that
\[\sup_\e \lambda_{0,\rho_{E^*}}\leq M_0<\infty.\] By Lemma \ref{Le:V0Lagrange}, $E^*_\e$ converges in $\mathscr C^{1,\alpha}$ to $\O$ (for any $\alpha<1$).  Now observe that, in fact, by a standard comparison principle, we may be satisfied to prove that $\lambda(\e)\to \infty$ where $\lambda(\e)$ is defined as 
\[ \lambda(\e):=\inf_{\p\in W^{1,2}_0(\O)}\frac{\int_\O |\n \p|^2}{\int_{\partial E_\e}\p^2}.\] This follows from the lower bound on $|\partial_\nu u_\e|$ given by Lemma \ref{Le:V0Lagrange}. The idea is then that, if the eigenvalues are uniformly bounded, then we should have on the one-hand (up to renormalisation)
\[ \int_{E_\e}\p_\e^2=1\] 
and, on the other hand, we should have enough regularity on $\p_\e$ to write 
\[ \int_{E_\e}\p_\e^2\sim \int_{\partial \O}\p_\e^2=0.\] This would be an obvious contradiction.
\end{remark}

\begin{proof}[Proof of Lemma \ref{Le:EigenvalueBlowup}]
We first observe that from Lemma \ref{Le:V0Lagrange} and the Rayleigh quotient formulation \eqref{Eq:RayleighRho} of $\lambda_{0,\rho_E}$  we have, for $V_0\in (1-\e_0;\e)$,
\[\lambda_{0,\rho_E}\geq  \delta_0\min_{u\in W^{1,2}_0(\O)\,, \int_{\partial E}u^2\neq 0}\frac{\int_\O |\n u|^2}{\int_{\partial E} u^2}.
\]

Now, argue by contradiction and assume  that there exists a sequence $\e\to0$, a constant $M_0>0$, and, for any $\e>0$, a critical set $E_\e$ such that $\mathds 1_{E_\e}\in \mathcal F(V_0)$ with $V_0\in (1-\e;1)$ that satisfies 
\[ \lambda_{0,\rho_{E_\e}}\leq M_0.\] Define, for any $\e>0$, 
\[ \lambda(\e):=\min_{u\in W^{1,2}_0(\O)\,, \int_{\partial E_\e}u^2\neq 0}\frac{\int_\O |\n u|^2}{\int_{\partial E_\e} u^2}.\] As already observed we have $\delta_0\lambda(\e)\leq \lambda_{0,\rho_{E_\e}}$ so that the sequence $\{\lambda(\e)\}_{\e \to 0}$ is uniformly bounded. In other words we have 
\begin{equation}\label{Eq:LambdaBounded} 0<\lambda(\e)\leq M_1<\infty\end{equation} for some suitable constant $M_1$. We define, for any $\e>0$, $\p_\e$ as the first eigenfunction associated with $\lambda(\e)$; in other words, $\p_\e$ solves the equation 
\begin{equation}\label{Eq:PhiEpsilon}
\begin{cases}
-\Delta \p_\e=0&\text{ in }\O\,, 
\\ \left\llbracket \frac{\partial \p_\e}{\partial \nu}\right\rrbracket=-\lambda(\e)\p_\e&\text{ on }\partial E_\e\,, 
\\ \p_\e\in W^{1,2}_0(\O)\,, 
\\ \int_{\partial E_\e} \p_\e^2=1.
\end{cases}
\end{equation}
Let us prove that if \eqref{Eq:LambdaBounded} holds then:
\begin{equation}\label{Eq:PhiEpsilonHolder}
\forall \beta \in (0;1)\,, \exists c_\beta\,, \sup_{\e>0}\Vert \p_\e\Vert_{\mathscr C^{0,\beta}}\leq c_\beta.
\end{equation}
First of all, from the uniform regularity estimates of Lemma \ref{Le:V0Lagrange} and from \cite[Theorem 18.34]{Leoni}, for any $p\in (1;+\infty)$, there exists a constant $C_p$ such that 
\[ \forall \e>0\,, \Vert \mathrm{Tr}_{E_\e}\Vert_{\mathcal L(W^{1,p}_0(\O);W^{1-\frac1p,p}(\partial E_\e))}\leq C_p.\] 

The idea is to use a bootstrap argument as well as the duality method of Stampacchia. Let show how to initialise the bootstrap. As 
\[ \int_\O |\n \p_\e|^2=\lambda(\e)\] it follows that
\[ \sup_\e \Vert \p_\e\Vert_{W^{\frac12,2}(\partial E_\e)}\leq C_0 M_1\] where $C_0$ is a trace embedding constant which only depends on $\sup_\e \mathrm{Lip}(E_\e)$.  Let $2'=\frac{2d}{2d-3}$ be the conjugate Sobolev exponent of 2 on the boundary.  Then we have 
\[ \sup_\e \Vert \p_\e\Vert_{L^{2'}(E_\e)}<C_1.\]

By a bootstrapping argument, we see that it suffices to prove the following regularity Lemma:
\begin{lemma}There exists a constant $C$ that only depends on $\mathrm{Lip}(E)$ and the covering number of $E$ such that the following holds: assume  $\psi$ solves 
\[ \begin{cases}-\Delta \psi=0\,, 
\\ \llbracket \partial_\nu \psi\rrbracket=-p\in L^q(\partial E)\end{cases}\] then \[ \Vert \psi\Vert_{W^{1,q}(\O)}\leq  C\Vert p\Vert_{L^q(\partial E)}.\]
\end{lemma}
\begin{proof}
Recall that, by standard elliptic estimates (see for instance \cite[Theorem 1.1]{Amrouche2020}), for any smooth vector field $F\in \mathscr C^\infty(\O)$, the solution $v_F$ of 
\[
\begin{cases}-\Delta v_F=\nabla\cdot F&\text{ in }\O\,, 
\\ v_F=0&\text{ on }\partial \O\end{cases}\] satisfies 
\[ \Vert v_F\Vert_{W^{1,r}(\O)}\leq D_r\Vert F\Vert_{L^r(\O)}\] for some constant $D_r=D(r,\O)$.

Now, consider a smooth vector field $F$, then there holds 
\[ \int_\O \langle \n \psi,F\rangle=\int_\O \langle \n \psi,\n v_F\rangle=\int_{\partial E} v_F p\leq C_E\Vert v_F\Vert_{W^{1,r}(\O)}\Vert p\Vert_{L^q(\partial E)}.\] Here, $C$ is the trace embedding constant; it only depends on $\mathrm{Lip}(E)$ and on the covering number of $\partial E$. Thus we get
\[ \int_\O \langle \n \psi,F\rangle\leq C_ED_r\Vert F\Vert_{L^r(\O)}\Vert p\Vert_{L^q(\partial E)}.\]Thus $\psi\in W^{1,q}(\O)$ and 
\[ \Vert \psi\Vert_{W^{1,q}(\O)}\leq C_ED_r \Vert p\Vert_{L^q(\partial E)}.\]
\end{proof}
Consequently we deduce by a bootstrap argument that, if \eqref{Eq:LambdaBounded} holds then, for any $p\in [1;+\infty)$, 
\[ \sup_{\e}\Vert \p_\e\Vert_{W^{1,p}(\O)}<\infty.\] We then conclude by the Sobolev embedding $W^{1,p}(\O)\hookrightarrow \mathscr C^{0,\beta}(\O)$. Thus \eqref{Eq:PhiEpsilonHolder} holds.

Finally, to derive a contradiction, we use once again the regularity results of Lemma \ref{Le:V0Lagrange}: from \eqref{Eq:PhiEpsilonHolder} and the fact that $\partial E_\e$ converges in $\mathscr C^1$ to $\partial \O$ we should have 
\begin{equation}
1=\int_{\partial E_\e} \p_\e^2\underset{\e \to 0}\rightarrow 0,
\end{equation}
an obvious contradiction. This concludes the proof of Lemma \ref{Le:EigenvalueBlowup}.
\end{proof}

We can now prove the uniform stability result of Proposition \ref{Pr:CriticalCoercivity}.
\begin{proof}[Proof of Proposition \ref{Pr:CriticalCoercivity}]
From the lower estimate \eqref{Eq:LowerEstimateLagrangian} we have, for any critical shape $E$ and any compactly supported vector field 
\[ \mathscr L_E''(E)[\Phi,\Phi]\geq \left(1-\frac{1}{\lambda_{0,\rho_E}}\right)\Vert \langle \Phi,\nu_E\rangle\Vert_{L^2(\partial E)}^2.\] From Lemma \ref{Le:EigenvalueBlowup}, there exists $\e_1>0$ such that, for any $V_0\in(1-\e_1;1)$,  for any critical shape $E$, 
\[ \lambda_{0,\rho_E}>2\] whence, fixing this $\e_1$, for any $V_0\in (1-\e_1;1)$, for any critical shape $E$, 
\[ \mathscr L_E''(E)[\Phi,\Phi]\geq 2\Vert \langle \Phi,\nu_E\rangle\Vert_{L^2(\partial E)}^2.\] This finishes the proof.
\end{proof}

\paragraph{Local quantitative inequalities around critical shapes.}
We now come to the two consequences of the previous analysis.  The first one is that critical shapes are $L^1$ isolated. The second one is that, whenever $V_0$ is close enough to 1, any critical shape is in fact a local $L^1$ minimiser. All the results of this section are straightforward adaptations of similar results in \cite{MRB2020}.

We begin with the first of these results.

\begin{lemma}\label{Le:CriticalIsolated}
Let $\e_1>0$ be given by Proposition \ref{Pr:CriticalCoercivity}. For any $V_0\in (1-\e_1;1)$ the critical shapes are isolated in the sense that 
\[ \inf_{E,E'\text{ critical, }E'\neq E}\left| E'\Delta E\right|=\delta(V_0)>0.\]
\end{lemma} This is proved exactly as \cite[Proposition 23]{MRB2020}.
\begin{lemma}\label{Le:CriticalMinimal}
Let $\e_1$ be given by Proposition \ref{Pr:CriticalCoercivity}. For any $V_0\in (1-\e_1;1)$, any critical shape $E$ is a local $L^1$ minimiser: there exist $c_E>0$ and $r_E>0$ such that, for any $f\in \mathcal F(V_0)$ with $\Vert f-\mathds 1_E\Vert_{L^1(\O)}\leq r_E$, there holds
\[ \mathscr E(f)-\mathscr E(\mathds 1_E)\geq c_E\Vert f-\mathds 1_E\Vert_{L^1(\O)}^2.\]
\end{lemma}
This is a straightforward adaptation of the proof of \cite[Theorem 1]{MRB2020}.
%
%

\color{black}
%

\subsection{Conclusion of the proof of Theorem \ref{Th:Dirichlet}}
\begin{proof}[Proof of Theorem \ref{Th:Dirichlet}]
Consider $\e\leq \e_1$ where $\e_1$ is given by Proposition \ref{Pr:CriticalCoercivity}. For any $V_0\in (1-\e;1)$, for any initialisation $f_0\in \mathcal F(V_0)$, let $\{f_k\}_{k\in \N}$ be the sequence generated by the thresholding algorithm. From Lemma \ref{Le:MultipleClosurePoints} the sequence $\{f_k\}_{k\in \N}$ has either one or infinitely many closure points. From  Lemma \ref{Le:ClosurePointBgbg} any closure point $f_\infty$ of $\{f_k\}_{k\in \N}$ is a critical point. From Lemma \ref{Le:CriticalIsolated}, critical points are isolated. Hence, there exists a unique closure point for the sequence $\{f_k\}_{k\in \N}$. Call it $f_\infty$. From Lemma \ref{Le:CriticalMinimal}, $f_\infty$ is a local minimiser. As any critical point is an extreme point of $\mathcal F(V_0)$ the sequence $\{f_k\}_{k\in \N}$ converges strongly in $L^1(\O)$ to $f_\infty$. The proof is concluded.
\end{proof}
\section{Conclusion, generalisations, obstructions}
We have thus established the convergence of the thresholding algorithm in three different cases, but under several restrictive assumptions we do not yet know how to bypass. Let us list below certain research questions we plan on tackling in the future.

\subsection{Lowering the volume constraint threshold}
A first major step that would need to be taken next is the extension of the convergence results established in this paper to all possible volume constraints $V_0$. Two remarks are in order in this case. The first one is that one would need an \emph{a priori} regularity theory. Although this is manageable, in the case of energetic functionals, in the two dimensional case thanks to \cite{Chanillo2008}, in higher dimensions, the solutions to the problem might no longer be regular enough to apply the shape derivative formalism put in place here. Second, the algorithm converges at best to a critical point. For low volume constraints, we can not guarantee that such critical points are local minimisers, let alone that they are regular enough (even in dimension 2) to use shape hessians, see \cite[Remark 3.20]{Chanillo2008bis}.

\subsection{Degenerate optimal control problems}
The first major underlying assumption is that all of our optimal control problems have extremal points of the admissible set of controls as solutions. Put otherwise, this amounts to requiring that the switch function has no flat zone. Now, in linear control problems for semilinear equations it is often the case that the optimal controls are not extreme points: the switch function has flat zones, and this requires a fine tuning of the thresholding method used to obtain satisfactory numerical simulations see \cite{Nadin_2020}. At this point it is unclear how one could tackle this question.

\subsection{Non-energetic bilinear optimal control problems}
There have been many contributions to the qualitative analysis of bilinear optimal control problems in recent years see \cite{MNP2021} and the references therein. The problem with this type of queries is that, despite the fact that the thresholding algorithm provides satisfactory simulations, several basic questions would need to be answered: the first one is the regularity of optimal controls. \emph{A priori} we can not hope for something better than a $\mathscr C^{1,\alpha}$ boundary, which is not high enough to apply second order shape derivatives arguments. It may be possible to push the methods of \cite{Chanillo_2000} to cover this setting but this is at the moment unclear.

\bibliographystyle{abbrv}

\bibliography{BiblioHessian}

\end{document}